\documentclass[sort&compress,3p]{elsarticle}
\usepackage{amsmath}
\usepackage{amssymb,mathtools}
\usepackage{array}
\usepackage{multirow}
\usepackage{graphicx}
\usepackage{bbm}
\usepackage{tikz}
\usepackage[normalem]{ulem}
\usepackage{dsfont}
\usepackage{float}
\usepackage{subfigure}
\usepackage{hyperref}
\usepackage{pdf14}
\usepackage{csl}

\usepackage{algpseudocode,algorithm}
\usepackage{tikz}

\newcommand{\K}{\mathcal{K}}
\newcommand{\R}{\mathbb{R}}
\newcommand{\I}{\mathbf{I}}
\newcommand{\V}{\mathbf{V}}
\newcommand{\W}{\mathbf{W}}
\newcommand{\Tb}{\mathbf{T}}
\newcommand{\J}{\mathbf{J}}

\newcommand{\A}{\mathbf{A}}

\newtheorem{theorem}{Theorem}
\newtheorem{remark}{Remark}
\newtheorem{lemma}{Lemma}
\newtheorem{corollary}{Corollary}
\newtheorem{definition}{Definition}
\newproof{proof}{Proof}

\begin{document}

\cslauthor{Ross Glandon, Paul Tranquilli and Adrian Sandu}
\cslyear{19}
\cslreportnumber{8}
\cslemail{rossg42@vt.edu, tranquilli1@llnl.gov, sandu@cs.vt.edu}
\csltitle{Biorthogonal Rosenbrock-Krylov time discretization methods}
\csltitlepage

\begin{frontmatter}

  \title{Biorthogonal Rosenbrock-Krylov time discretization methods}
  
  \author[csl]{Ross Glandon\corref{cor}}
  \ead{rossg42@vt.edu}
  \author[csl,llnl]{Paul Tranquilli}
  \ead{tranquilli1@llnl.gov}
  \author[csl]{Adrian Sandu}
  \ead{sandu@cs.vt.edu}
  
  \address[csl]{Computational Science Laboratory, Department of Computer Science, Virginia Tech. 2202 Kraft Drive, Blacksburg, Virginia 24061}
  \address[llnl]{Lawrence Livermore National Laboratory. 7000 East Avenue, Livermore, California 94550}
  \cortext[cor]{Corresponding author}

\begin{abstract}
Many scientific applications require the solution of large initial-value problems, such as those produced by the method of lines after semi-discretization in space of partial differential equations. The computational cost of implicit time discretizations is dominated by the solution of nonlinear systems of equations at each time step. In order to decrease this cost, the recently developed Rosenbrock-Krylov (ROK) time integration methods extend the classical linearly-implicit Rosenbrock(-W) methods, and make use of a Krylov subspace approximation to the Jacobian computed via an Arnoldi process. Since the ROK order conditions rely on the construction of a single Krylov space, no restarting of the Arnoldi process is allowed, and the iterations quickly become expensive with increasing subspace dimensions. This work  extends the ROK framework to make use of the Lanczos biorthogonalization procedure for constructing Jacobian approximations. The resulting new family of methods is named biorthogonal ROK (BOROK). The Lanczos procedure's short two-term recurrence allows BOROK methods to utilize larger subspaces for the Jacobian approximation, resulting in increased numerical stability of the time integration at a reduced computational cost. Adaptive subspace size selection and basis extension procedures are also developed for the new schemes. Numerical experiments show that for stiff problems, where a large subspace used to approximate the Jacobian is required for stability, the BOROK methods outperform the original ROK methods.
\end{abstract}


\end{frontmatter}

\section{Introduction}
In this paper we are concerned with numerical solutions of the initial-value problem:
\begin{equation}
\label{eqn:ivp}
  \frac{dy}{dt} = f(t, y), \quad t_0 \leq t \leq t_F, \quad y(t_0) = y_0; \quad y(t) \in \R^N, \quad f:\R\times\R^N \rightarrow \R^N.
\end{equation}
A large number of scientific and engineering simulations rely on problems of this structure, with examples ranging from densely coupled systems of ordinary differential equations (ODEs) such as those arising from chemical kinetics, aerosol dynamics, or multibody dynamics \cite{Sandu_2003_aerosolFramework,Sandu_2010_KPP22_TLM_ADJ,Sandu_2009_ParameterExplicit} to large sparse systems produced by the method of lines after semi-discretization in space of partial differential equations (PDEs) such as fluid dynamics or electromagnetics \cite{Munson_2013_aa,Ulaby_2010_aa}.

Time integration methods for solving initial-value problems \eqref{eqn:ivp} are traditionally categorized as either explicit or implicit. Explicit methods use previously obtained information in order to calculate the solution at the following timestep, are simple in structure and inexpensive, but require a problem-dependent timestep size restriction. Implicit methods relax the timestep size restriction by incorporating future information in the computation of the solution, but adds the requirement to solve a nonlinear system of equations at each timestep, considerably increasing the per-step cost of the method.

Much current research seeks to blur the lines between implicit and explicit time integration methods. IMEX methods \cite{Ascher1995,Ascher1997,Frank1997,Pareschi2005,Hundsdorfer2007,Boscarino2009,Sandu_2012_ICCS-IMEX,Sandu_2016_highOrderIMEX-GLM,Sandu_2010_extrapolatedIMEX,Sandu_2014_IMEX-RK,Sandu_2014_IMEX_GLM_Extrap,Sandu_2014_IMEX-GLM,Sandu_2015_IMEX-TSRK,Sandu_2015_Stable_IMEX-GLM} treat a split IVP with a pair of coupled methods, one implicit, the other explicit. Multirate methods \cite{Rice1960,Andrus1979,Andrus1993,Gear1984,Engstler1997,Gunther2001,Sandu2009,Constantinescu2013,Gunther2016,Sandu_2019_MR-GARK_High-Order,Sandu_2018_MRI-GARK-coupled,Sandu_2018_MRI-GARK,Sandu_2007_MR_RK2} use different timestep sizes to solve different components of the problem. W-methods (first introduced in the context of Rosenbrock schemes \cite{Steihaug1979,Rang2005,Rahunanthan2010,Schmitt1995,Wensch2005}) use the problem Jacobian, $f_y(t, y) = \J$, explicitly in the computational process (in either a linearly-implicit or matrix exponential formula) and use additional order conditions to eliminate the errors associated with using an approximate Jacobian. K-methods \cite{Sandu_2014_ROK,Sandu_2014_expK,Sandu_2019_EPIRKW} are a subclass of W-methods that makes use of a particular Krylov-based approximation to the problem Jacobian, $\J$, resulting in simpler order conditions and some potential computational advantage. 

Rosenbrock-Krylov (ROK) methods, the first K-methods, were originally proposed in \cite{Sandu_2014_ROK} as an extension to Rosenbrock-W methods, making use of an Arnoldi process to produce the K-method's Jacobian approximation. Like W-methods, K-methods decouple accuracy and stability concerns, with an order $p$ ROK method only requiring a Krylov space of size $p$ to retain full order of accuracy. However, for large stiff problems, a much larger Krylov space $m \gg p$ may be required to achieve reasonable stability. In this case, the Arnoldi process used in the ROK method may not be computationally feasible, due to the $m$-term recurrence appearing in the orthonormalization of the Krylov basis matrix. In order to overcome this challenge, this paper extends ROK methods to make use of biorthogonal sequences (which requires a fixed 3-term recurrence to produce the Krylov basis matrices) in the calculation of the Krylov-based Jacobian approximation matrix. We call the new schemes Biorthogonal Rosenbrock-Krylov (BOROK) methods.

The remainder of the paper is organized as follows. Section \ref{sec:ROK} reviews the ROK time integration methods. Section \ref{sec:BOROK} derives the BOROK methods that extend ROK to use the Lanczos biorthogonalization procedure. Sections \ref{sec:BOROKadapt} and \ref{sec:BOROKext} discuss practical stability-enhancing tools used in implementations of BOROK. Section \ref{sec:results} reports the results of numerical experiments that compare the new methods with the base ROK methods. Section \ref{sec:conclusion} draws the conclusions of this work.

\section{Rosenbrock-Krylov methods}
\label{sec:ROK}

Rosenbrock-Krylov methods \cite{Sandu_2014_ROK} are a recent extension to the classical Rosenbrock (ROS) and Rosenbrock-W (ROW) methods \cite[Chapter IV.7]{Hairer1996}. ROS methods are linearly-implicit schemes requiring only a single linear system solve per stage. ROW methods are an extension of ROS that accounts for approximate Jacobian matrices, at the cost of vastly more order conditions for high order methods. The general form for an autonomous Rosenbrock method is \cite[Definition 7.1]{Hairer1996}:
\begin{subequations}
\label{eqn:ros-general-form}
\begin{align}
  F_i & = \displaystyle f\left(y_n + \sum_{j=1}^{i-1}\alpha_{i,j} k_j\right), \label{eqn:ros-func-eq}\\
  k_i & = \displaystyle \varphi\left(h\gamma\A\right) \left(h F_i + h\A \sum_{j=1}^{i-1} \gamma_{i,j} k_j \right), \label{eqn:ros-stage-eq}\\
  y_{n+1} & = \displaystyle y_n + \sum_{i=1}^s b_i k_i, \label{eqn:ros-step-eq}
\end{align}
\end{subequations}
where $\A$ is the exact Jacobian $\J(y_n)$ for Rosenbrock methods, or an arbitrary approximation for Rosenbrock-W methods. The matrix function $\varphi(z) = 1/(1 - z)$ produces Rosenbrock(-W) methods, while $\varphi(z) = 1$ gives explicit Runge-Kutta schemes, and $\varphi(z) = \left(e^{z} - 1\right)/z$ gives exponential methods.

Throughout this paper, we restrict formulas to the autonomous form of the methods for simplicity, as the nonautonomous form can be easily derived by applying the methods to an augmented system as described in \cite[Equations 2.11 and 2.12]{Sandu_2014_ROK}

Whereas ROW methods admit any arbitrary approximation of the Jacobian, Rosenbrock-Krylov (ROK) methods restrict themselves to a particular, computationally-favorable, Krylov subspace approximation of the Jacobian. At each step ROK methods construct a Krylov space from the ODE right-hand side (RHS) function $f$, and its Jacobian $\J$, both evaluated at the current time and approximate numerical solution $(t_n,y_n)$:
\[
  \K_m\left(\J, f\right) = \text{span}\left(f, \J f, \J^2 f, \dots, \J^{m-1}f\right) = \text{span}\left(\V_m\right).
\]
The Krylov space described by the orthonormal basis matrix $\V_m$ computed via Arnoldi's method \cite[Chapter 6.3]{Saad2003} (or the symmetric Lanczos algorithm \cite[Ch. 6.6]{Saad2003} for symmetric Jacobians). This basis matrix is then used to define the following approximate Jacobian:
\[
  \A \coloneqq \V_m\, \V_m^T\, \J\, \V_m\, \V_m^T.
\]
This approximation leads to both a meaningful reduction in the number of order conditions compared to ROW methods, and a computationally convenient reduced form. However, in the case of a non-symmetric Jacobian for very large stiff problems, the dimension $m$ of the Krylov space needs to be large. In this case the $m$-term recurrence for the orthonormalization of $\V_m$, which has a cost $\mathcal{O}(m^2)$, can become infeasible. In this paper we extend ROK methods to allow the use of Lanczos biorthogonalization, which has a short three-term recurrence, and is computationally favorable when large subspace dimensions are required. The new schemes are named Biorthogonal Rosenbrock-Krylov methods.

\section{Biorthogonal Rosenbrock-Krylov (BOROK) methods}
\label{sec:BOROK}

\subsection{Lanczos biorthogonalization and the approximate Jacobian}

Lanczos biorthogonalization (Algorithm \ref{alg:lanczos_biorth}) produces a pair of biorthogonal basis matrices, $\V_m$ and $\W_m$, for two $m$-dimensional Krylov subspaces:
\begin{align*}
\K_m\left(\J, f\right) & = \text{span}\left(f, \J f, \J^2 f, \dots, \J^{m-1}f\right) = \text{span}\left(\V_m\right), \\
\K_m\left(\J^T, f\right) & = \text{span}\left(f, \J^T f, (\J^T)^2 f, \dots, (\J^T)^{m-1}f\right) = \text{span}\left(\W_m\right).
\end{align*}
While in general the two Krylov spaces are defined using different vectors, here the same vector $f$ is used for both spaces. In addition, the Lanczos biorthogonalization algorithm produces the tridiagonal matrix $\Tb_m \in \R^{m \times m}$ which, together with $\V_m$ and $\W_m$, satisfies the following biorthogonal relations \cite[Proposition 7.1]{Saad2003}:
\begin{subequations}
\label{eqn:biorthoprop}
\begin{align}
  \V_m^T\, \W_m & = \mathbf{I}_m, \label{eqn:biortho-id}\\
  \J\,\V_m & = \V_m\, \Tb_m + \theta_{m+1}\,v_{m+1}\,e_m^T, \label{eqn:biortho-jv}\\
  \J^T\W_m & = \W_m\, \Tb_m^T + \beta_{m+1}w_{m+1}\,e_m^T, \label{eqn:biortho-jtw}\\
  \Tb_m & = \W_m^T\, \J\, \V_m. \label{eqn:biortho-tmat}
\end{align}
\end{subequations}

\begin{algorithm}[ht]
\caption{Lanczos Biorthogonalization \cite[Algorithm 7.1]{Saad2003}}
\label{alg:lanczos_biorth}
\begin{algorithmic}[1]
   \State Choose two vectors $v_1$, $w_1$ such that $(v_1, w_1) = 1$.
   \State Set $\beta_1 = \theta_1 = 0$, $w_0 = v_0 = 0$
   \For{$j = 1,2,\dots,m$}
     \State $\kappa_j = \left(Av_j, w_j\right)$
     \State $\hat{v}_{j+1} = A v_j - \kappa_j v_j - \beta_j v_{j-1}$
     \State $\hat{w}_{j+1} = A^T w_j - \kappa_j w_j - \theta_j w_{j-1}$
     \State $\theta_{j+1} = \left|\left(\hat{v}_{j+1}, \hat{v}_{j+1}\right)\right|^{1/2}$
     \State $\beta_{j+1} = \left(\hat{v}_{j+1}, \hat{w}_{j+1}\right) / \theta_{j+1}$
     \State $w_{j+1} = \hat{w}_{j+1} / \beta_{j+1}$
     \State $v_{j+1} = \hat{v}_{j+1} / \theta_{j+1}$
   \EndFor
\end{algorithmic}
\end{algorithm}

%
%

\subsection{BOROK method formulation}

\begin{definition}[BOROK methods]
A BOROK method defined in the full space is the ROW method \ref{eqn:ros-general-form}
%
with the matrix function $\varphi(z) = 1/(1 - z)$,
where the Jacobian approximation $\A$ is the following biorthogonal projection \eqref{eqn:krylovapproxmat} of the exact Jacobian $\J(y_n)$
\begin{equation}
  \label{eqn:krylovapproxmat}
  \A \coloneqq \V_m\,\Tb_m \,\W_m^T = \V_m\, \W_m^T\, \J\, \V_m\, \W_m^T.
\end{equation}
\end{definition}

We establish several properties of the approximate Jacobian \eqref{eqn:krylovapproxmat}.
\begin{lemma}[Powers of the approximate Jacobian]
\label{lem:kmatpowers}
  \begin{equation}
   \label{eqn:kmatpowers}
   \A^i = \V_m\, \Tb_m^i\, \W_m^T.
  \end{equation}
\end{lemma}
\begin{proof}
  For the base case $i = 1$, the lemma is true by the definition \eqref{eqn:krylovapproxmat} of $\A$. Continuing with the inductive proof, assuming the result holds for $i-1$,  the result for $i$ follows:
  \begin{align*}
    \A^{i-1} & = \V_m\, \Tb^{i-1}\, \W_m^T, \\
    \A^i & = \V_m\, \Tb_m^{i-1}\, \underbrace{\W_m^T\, \V_m}_{\mathbf{I}_m}\, \Tb_m\, \W_m^T = \V_m\, \Tb_m^{i}\, \W_m^T.
  \end{align*}
  \qed
\end{proof}

\begin{lemma}[Matrix functions of the approximate Jacobian]
  \label{lem:phifuncs}
  An analytic matrix function $\varphi$ of the full space approximate Jacobian $\A$ can be written in terms of the matrix function of the reduced matrix $\Tb_m$ as follows:
  \begin{equation}
  \label{eqn:phi}
    \varphi\left(h\gamma\A\right) = \varphi(0)\left(\I_m - \V_m\,\W_m^T\right) + \V_m\, \varphi\left(h\gamma\,\Tb_m\right) \W_m^T.
  \end{equation}
\end{lemma}
\begin{proof}
The analytic matrix function $\varphi(z)$ has a Taylor series about $z = 0$:
  \[
    \varphi(z) = \sum_{i=0}^{\infty} c_i \frac{z^i}{i!}.
  \]
  Applying $\varphi(z)$ to the scaled matrix $h\gamma\A$ gives
  \[
    \varphi\left(h\gamma\A\right) = c_0\I_m + \sum_{i=1}^\infty c_i\frac{\left(h\gamma\right)^i}{i!} \A^i,
  \]
  and using Lemma \ref{lem:kmatpowers} we obtain:
  \begin{equation}
    \label{eqn:phifull}
    \varphi\left(h\gamma\A\right) = c_0\I_m + \sum_{i=1}^\infty c_i\frac{\left(h\gamma\right)^i}{i!} \V_m\, \Tb_m^i\, \W_m^T.
  \end{equation}
  We can perform a similar expansion of functions of the reduced space matrix:
  \begin{equation}
    \label{eqn:phired}
    \V_m\, \varphi\left(h\gamma\,\Tb_m\right) \W_m^T = c_0\V_m\,\W_m^T + \sum_{i=1}^\infty c_i\frac{\left(h\gamma\right)^i}{i!} \V_m\, \Tb_m^i\, \W_m^T.
  \end{equation}
 Taking the difference between equations \eqref{eqn:phifull} and \eqref{eqn:phired} gives
  \[
    \varphi\left(h\gamma\A\right) - \V_m\, \varphi\left(h\gamma\,\Tb_m\right) \W_m^T = c_0\left(\I_N - \V_m\,\W_m^T\right),
  \]
  and noting that $c_0 = \varphi(0)$ from the Taylor series expansion we obtain the result \eqref{eqn:phi}.
  \qed
\end{proof}

\begin{definition}[Reduced space projections]
  \label{def:redprojections}
  We define the reduced space projections $\psi_i$ and $\lambda_i$ of the full space vectors $k_i$ and $F_i$:
  \begin{equation}
    \lambda_i := \W_m^T\, k_i, \qquad \psi_i := \W_m^T\, F_i.
  \end{equation}
  The vectors $k_i$ and $F_i$ decompose into one component in $\K_m(\J,f)$ and another component orthogonal to $\K_m(\J^T,f)$:
  \begin{equation}
  \label{eqn:decomposition}
    k_i = \underbrace{\V_m \lambda_i}_{\in \K_m(\J, f)} + \underbrace{\mu_i,}_{\in \K_m^\perp(\J^T, f)} \qquad F_i = \underbrace{\V_m \psi_i}_{\in \K_m(\J, f)} + \underbrace{\delta_i.}_{\in \K_m^\perp(\J^T, f)}
  \end{equation}
\end{definition}

\begin{theorem}[Reduced-space implementation of BOROK methods]
  \label{thm:borok-general-form}
  With the approximate Jacobian defined as in equation \eqref{eqn:krylovapproxmat}, the stage vectors $k_i$:
  \begin{equation}
  \label{eqn:borok_stage}
  k_i  =  \displaystyle \varphi\left(h\gamma\A\right) \left(h F_i + h\A \sum_{j=1}^{i-1} \gamma_{i,j} k_j \right)
  \end{equation}
 can be computed using only matrix functions evaluated in the reduced space as follows:
  \begin{subequations}
    \begin{align}
      \psi_i & = \W_m^T F_i, \\
      \lambda_i & = \varphi\left(h\gamma\,\Tb_m\right) \left(h \psi_i + h\Tb_m \sum_{j=1}^{i-1} \gamma_{i,j}\, \lambda_j \right), \\
      k_i & = \V_m\,\lambda_i + h \left(F_i - \V_m\,\psi_i\right).
    \end{align}
  \end{subequations}
\end{theorem}
\begin{proof}
  We begin the proof by substituting the $k_i$ and $F_i$ decompositions from \eqref{eqn:decomposition} into the stage equation \eqref{eqn:borok_stage}, using the definition of the approximate Jacobian \eqref{eqn:krylovapproxmat}, and applying Lemma \ref{lem:phifuncs} with $\varphi(z) = 1/(1-z)$. This gives:
  \begin{align*}
    k_i & = \V_m \lambda_i + \mu_i \\
        & = h \Big[\left(\I_N - \V_m\W_m^T\right) + \V_m \varphi\left(h\gamma\,\Tb_m\right)\W_m^T\Big] \left[\V_m\,\psi_i + \delta_i + \V_m\,\Tb_m\W_m^T \sum_{j=1}^{i-1} \gamma_{i,j} \left(\V_m\,\lambda_j + \mu_j\right)\right].
  \end{align*}
  Since $\W_m^T\,\V_m = \I_m$ from equation \eqref{eqn:biortho-id}, and using the fact that $\delta_i$ and $\mu_i$ are orthogonal to $\W_m$ by definition \eqref{eqn:decomposition}, the expression simplifies to:
  \[
    k_i = h\,\delta_i + h\,\V_m\,\varphi\left(h\gamma\,\Tb_m\right) \left[\psi_i + \Tb_m \sum_{j=1}^{i-1} \gamma_{i,j}\, \lambda_i\right].
  \]
The stage vector component orthogonal to $\K_m(\J^T, f)$ \eqref{eqn:decomposition} is:
  \[
    \mu_i = \left(\I_N - \V_m\W_m^T\right)\,k_i = h\,\delta_i.
  \]
Similarly, the stage vector component in $\K_m(\J, f)$ \eqref{eqn:decomposition} is:
  \[
    \lambda_i = \W_m^T\,k_i = h\,\varphi\left(h\gamma\,\Tb_m\right)\, \Big[\psi_i + \Tb_m \sum_{j=1}^{i-1} \gamma_{i,j}\, \lambda_j\Big].
  \]
  Recombine the elements of $k_i$ yields:
  \[
    k_i = \V_m\, \lambda_i + h\,\delta_i = \V_m\, \lambda_i + h\, \left(F_i - \V_m\,\psi_i\right),
  \]
  with the last expansion coming from the definition of $\delta_i$ in \eqref{eqn:decomposition}.
  \qed
\end{proof}


\begin{definition}[BOROK methods in reduced form]
\label{def:borok-reduced-form}
From theorem \ref{thm:borok-general-form}, the reduced form of an \emph{s}-stage BOROK method is:
\begin{subequations}
\label{eqn:borok-reduced-form}
\begin{align}
 F_i & = f\left(y_n + \sum_{j=1}^{i-1}\alpha_{i,j} k_j\right) \label{eqn:redstagef}\\
 \psi_i & = \W_m^T F_i \label{eqn:redstagepsi}\\
 \lambda_i & = \displaystyle\left(\I_m - h\gamma\,\Tb_m\right)^{-1} \left(h \psi_i + h\Tb_m \sum_{j=1}^{i-1} \gamma_{i,j}\, \lambda_j \right) \label{eqn:redstageeq}\\
 k_i & = \V_m \lambda_i + h \left(F_i - \V_m \psi_i\right) \label{eqn:redstagek}\\
 y_{n+1} & = y_n + \sum^s_{i=1} b_i k_i \label{eqn:redstepeq}
\end{align}
\end{subequations}

One step of a BOROK method in reduced space form is shown in Algorithm \ref{alg:borok_onestep}.
\end{definition}

\begin{algorithm}[ht]
\caption{One step of the reduced-space form Biorthogonal Rosenbrock-Krylov method.}
\label{alg:borok_onestep}
\begin{algorithmic}[1]
   \State Compute $\V_m$, $\W_m$, and $\Tb_m$ using the Lanczos biorthogonalization procedure.
   \For{$i = 1,\dots,s$}
     \State $F_i = \displaystyle f\left(y_n + \sum_{j=1}^{i-1}\alpha_{i,j} k_j\right)$
     \State $\psi_i = \W_m^T F_i$
     \State $\lambda_i = \displaystyle\left(\I_m - h\gamma\,\Tb_m\right)^{-1} \left(h \psi_i + h\Tb_m \sum_{j=1}^{i-1} \gamma_{i,j}\, \lambda_j \right)$
     \State $k_i = \V_m \lambda_i + h \left(F_i - \V_m \psi_i\right)$
   \EndFor
   \State $y_{n+1} = y_n + \sum^s_{i=1} b_i k_i$
\end{algorithmic}
\end{algorithm}

\begin{remark}
The implementation of the reduced form requires explicit storage of the $\W_m$ basis matrix so that the right-hand side evaluations $F_i$ for stages 2 to \emph{s} can be projected into the reduced space to produce $\psi_i$. This appears to prevent the optimizations made by BiCGSTAB, which does not compute any of the vectors for $\W_m$ at all. There are a wide variety of existing block iterative linear system solvers in the literature which are capable of solving multiple linear systems against the same matrix \cite{Guennouni2003, Heyouni2005}. However these block solvers usually require all linear system right-hand side vectors at the same time (which is not possible here due to the stage-to-stage dependency of $F_i$ on $k_{i-1}$). There are a few solvers that are formulated for sequentially dependent RHS vectors, but these seed system methods are based on CG or GMRES and do not make use of Lanczos biorthogonalization \cite{Chan1997,Parlett1980,Saad1987,Vandervorst1987}.
\end{remark}

\subsection{Order conditions}

We would like to make full use of the existing K-method order condition theory and the ROK method coefficients, as derived in \cite{Sandu_2014_ROK}. This derivation hinges on the recoloring of the TW-trees used to produce the order conditions for Rosenbrock-W methods (a type of two-color Butcher-trees \cite[Section IV.7]{Hairer1996} used to represent individual elementary differentials in a Taylor series expansion corresponding to either the exact Jacobian $\J$ or the approximate Jacobian $\A$) into the TK-trees \cite[Definition 3.3]{Sandu_2014_ROK}. This recoloring of trees hinges on the following lemma demonstrating that powers of the Krylov approximation matrix $\A$ applied to $f$ are equivalent to applying powers of the exact Jacobian $\J$ to $f$.

\begin{lemma}
  \label{lem:ordercondsat}
  For any $0 \leq i \leq m-1$, it holds that
  \begin{equation}
   \A^i f = \J^i f
  \end{equation}
  where $m = \text{dim}\left(\K_m(\J,f)\right)$.
\end{lemma}
\begin{proof}
  Using the Definition \ref{eqn:krylovapproxmat} we have:
  \[\A^i f = \left(\V_m\, \W_m^T\, \J\, \V_m\, \W_m^T\right)^i f.\]
  Because the matrix $\V_m\, \W_m^T$ is an oblique projector into the span of $\K_m(\J,f)$, which contains both $f$ and $\J f$ for $m > 1$, we can easily prove the base case $i = 1$:
  \[
    \A f = \V_m\, \W_m^T\, \J\, \underbrace{\V_m\, \W_m^T f}_{f} = \V_m\, \W_m^T\, \J f = \J f.
  \]
  We then make the induction hypothesis that $\A^{i-1} f = \J^{i-1} f$ holds for some $2 \leq i \leq m-1$. The result for $i$ is:
  \[
    \A^i f = \A\, \A^{i-1} f = \A\, \J^{i-1}f = \V_m\, \W_m^T\, \J\, \underbrace{\V_m\, \W_m^T\, \J^{i-1} f}_{\J^{i-1}f} = \V_m\, \W_m^T\, \J^i f = \J^i f,
  \]
because both $\J^{i-1}f$ and $\J^i f$ are in $\K_m(\J,f)$.
  \qed
\end{proof}

Lemma \ref{lem:ordercondsat} allows to prove that the recoloring of the TW-trees into TK-trees for the Jacobian approximation \eqref{eqn:krylovapproxmat} is the same as for the Krylov projection approximation discussed in \cite{Sandu_2014_ROK}.

\begin{lemma}[Recoloring of linear subtrees using the approximate Jacobian \eqref{eqn:krylovapproxmat}]
When the Krylov matrix approximation \eqref{eqn:krylovapproxmat} is used, all linear TW-trees of order $k \leq M$ correspond to a single elementary differential, regardless of the color of their nodes.
\end{lemma}
\begin{proof}
The proof is identical to that given in \cite[Lemma 3.2]{Sandu_2014_ROK}, making use of repeated application of Lemma \ref{lem:ordercondsat}.
\qed
\end{proof}

\begin{corollary}
BOROK methods have identical order conditions as ROK methods. Consequently, BOROK schemes can use directly  previously derived sets of coefficients $\mathbf{\alpha}$, $\mathbf{\gamma}$, and $b$ for ROK methods, such as those from \cite{Sandu_2014_ROK, Wu_2016_ROK4E}.
\end{corollary}
%
 
\section{Basis size adaptivity for BOROK methods using stage residuals}
\label{sec:BOROKadapt}

The numerical stability of K-methods depends on the dimension of the Krylov space on which the Jacobian is projected. In particular, for $m \ge N$, the projected Jacobian equals the exact Jacobian, the function $\varphi(z) = 1/(1-z)$ is evaluated exactly, and the unconditional stability of the base Rosenbrock scheme is recovered. The stability study is more complex in the practical situation where $m \ll N$.

It was shown that the numerical stability of ROK methods can be greatly enhanced by enforcing a convergence condition on the residual of the linear systems in each stage equation \cite{Tranquilli2019}. Here we extend this approach to BOROK methods by first expressing the stage residuals in terms of the new approximate Jacobian and the biorthogonal relations \eqref{eqn:biorthoprop}, then modifying the Lanczos biorthogonalization procedure (Algorithm \ref{alg:lanczos_biorth}) to incorporate convergence checks.

\begin{theorem}[Stage linear system residuals for BOROK methods]
  \label{thm:stgresidual}
  The residual of the linear system at the \emph{i}th stage of an \emph{s}-stage biorthogonal ROK method is:
\begin{equation}
\label{eqn:stage-residual}
r_{m;i} = - h^2\, \J\, \sum_{j=1}^i \gamma_{i,j}\, \left(F_j - \V_m \psi_j\right) - h \theta_{m+1}\,v_{m+1}\,e_m^T\, \sum_{j=1}^i \gamma_{i,j}\,\lambda_j.
\end{equation}
\end{theorem}
\begin{proof}
  First, we recall the full space form of the stage equation \eqref{eqn:ros-stage-eq}, compensated with a residual, $r_{m;i}$, due to the inexact solve:
  \begin{equation*}
    \left(\I_N - h\gamma\,\J\right) k_i = h F_i + h\,\J \sum_{j=1}^{i-1}\gamma_{i,j}k_j + r_{m;i}.
  \end{equation*}
  Now, we isolate $r_{m;i}$ and substitute the $k_i$ formula from \eqref{eqn:redstagek}
  \[
    r_{m;i} = \left(\I_N - h\gamma\,\J\right)\left(\V_m\,\lambda_i + h F_i - h\V_m\,\psi_i\right) - h F_i - h\,\J \sum_{j=1}^{i-1} \gamma_{i,j} \left(\V_m\,\lambda_j + h F_j - h\V_m\,\psi_j\right).
  \]
  Expanding and canceling the $h F_i$ terms leads to:
  \begin{equation*}
  \begin{split}
    r_{m;i} &= \V_m\,\lambda_i - h\V_m\,\psi_i - h\gamma\,\J\,\V_m\,\lambda_i - h^2\gamma\,\J F_i + h^2\gamma\,\J\,\V_m\,\psi_i \\
 &\quad   - h^2\,\J \sum_{j=1}^{i-1} \gamma_{i,j}\, F_j - h\,\J\,\V_m \sum_{j=1}^{i-1} \gamma_{i,j}\left(\lambda_j - h\psi_j\right).
 \end{split}
  \end{equation*}
  Equation \eqref{eqn:biortho-jv}
  \[
    \J\,\V_m = \V_m\,\Tb_m + \theta_{m+1}\,v_{m+1}\,e_m^T
  \]
  can then be used to further expand the $\lambda$ terms to:
  \begin{multline*}
    r_{m;i} = \V_m\,\lambda_i - h\V_m\,\psi_i - h\gamma\,\V_m\,\Tb_m\,\lambda_i - h\gamma\,\theta_{m+1}\,v_{m+1}\,e_m^T\,\lambda_i - h^2\gamma\,\J F_i + h^2\gamma\,\J\,\V_m\,\psi_i \\
    - h^2\,\J \sum_{j=1}^{i-1} \gamma_{i,j}\, F_j - h\V_m\,\Tb_m \sum_{j=1}^{i-1} \gamma_{i,j}\, \lambda_j - h\,\theta_{m+1}\,v_{m+1}\,e_m^T \sum_{j=1}^{i-1} \gamma_{i,j}\, \lambda_j + h^2\,\J\,\V_m \sum_{j=1}^{i-1} \gamma_{i,j}\, \psi_j.
  \end{multline*}
  Rearanging and factoring leads to
  \begin{multline*}
    r_{m;i} = \V_m\underbrace{\left[\left(\I_m - h\gamma\,\Tb_m\right)\lambda_i - h\psi_i - h\Tb_m \sum_{j=1}^{i-1} \gamma_{i,j}\, \lambda_j\right]}_{=\,0} - h\gamma\,\theta_{m+1}\,v_{m+1}\,e_m^T\,\lambda_i - h^2\gamma\,\J F_i + h^2\gamma\,\J\,\V_m\,\psi_i \\
    - h^2\,\J \sum_{j=1}^{i-1} \gamma_{i,j}\, F_j - h\,\theta_{m+1}\,v_{m+1}\,e_m^T \sum_{j=1}^{i-1} \gamma_{i,j}\, \lambda_j + h^2\,\J\,\V_m \sum_{j=1}^{i-1} \gamma_{i,j}\, \psi_j,
  \end{multline*}
  where we recognize the bracketed terms as satisfying the reduced stage equation \eqref{eqn:redstageeq} from definition \ref{def:borok-reduced-form}. Next, we collect all terms into sums:
  \[
    r_{m;i} = h^2\,\J\,\V_m \sum_{j=1}^{i} \gamma_{i,j}\,\psi_j - h^2\,\J \sum_{j=1}^{i} \gamma_{i,j}\, F_j - h\,\theta_{m+1}\,v_{m+1}\,e_m^T \sum_{j=1}^i \gamma_{i,j}\, \lambda_j.
  \]
  Finally, regrouping the terms gives the desired result.
  \qed
\end{proof}

\begin{corollary}[First stage residual of a BOROK method]
  \label{cor:stageoneres}
  The linear system residual for the first stage of a biorthogonal ROK method has the form:
  \[
    r_{m;1} = - h\gamma\,\theta_{m+1}\,v_{m+1}\,e_m^T\,\lambda_1,
  \]
  with the norm
  \[
    \left\|r_{m;1}\right\| = \left|h\gamma\,\theta_{m+1} e_m^T\,\lambda_1\right| \left\|v_{m+1}\right\|.
  \]
  Further, one can choose $\left\|v_{m+1}\right\| = 1$ in the Lanczos biorthogonalization procedure.
\end{corollary}

The full stage residual \eqref{eqn:stage-residual} from Theorem \ref{thm:stgresidual} contains a full-space term requiring matrix-vector multiplication against the full Jacobian, as well as information from all previous stage function evaluations $F_j$ and reduced space solutions $\psi_j$. Thus, it would be very costly to use of the residuals from stages $2$ to $s$ in the construction of the Krylov bases. Instead, using Corollary \ref{cor:stageoneres}, we modify the Lanczos biorthogonalization procedure in Algorithm \ref{alg:lanczos_biorth} to test for convergence of only the first stage residual. This approach leads to Algorithm \ref{alg:lanczos_biorth_conv}.

\begin{algorithm}[ht]
\caption{Lanczos biorthogonalization with convergence test of the first residual.}
\label{alg:lanczos_biorth_conv}
\begin{algorithmic}[1]
   \State Choose two vectors $v_1$, $w_1$ such that $(v_1, w_1) = 1$.
   \State Set $\beta_1 = \delta_1 = 0$, $w_0 = v_0 = 0$
   \For{$j = 1,2,\dots$ until convergence}
     \State $\kappa_j = \left(Av_j, w_j\right)$
     \State $\hat{v}_{j+1} = A v_j - \kappa_j v_j - \beta_j v_{j-1}$
     \State $\hat{w}_{j+1} = A^T w_j - \kappa_j w_j - \theta_j w_{j-1}$
     \State $\theta_{j+1} = \left|\left(\hat{v}_{j+1}, \hat{v}_{j+1}\right)\right|^{1/2}$
     \State $\beta_{j+1} = \left(\hat{v}_{j+1}, \hat{w}_{j+1}\right) / \theta_{j+1}$
     \If{$j \geq 4$}
       \State $\lambda_1 = \left(\I_j - h\gamma\,\Tb_j\right)^{-1} \left(h \W_j^T F_i\right)$
       \State $\left\|r_{1;j}\right\| = \left|h\gamma\,\theta_{j+1} e_j^T\,\lambda_1\right|$
       \If{$\left\|r_{1;j}\right\| \leq \text{TOL}$}
         \State \textbf{break}
       \EndIf
     \EndIf
     \State $w_{j+1} = \hat{w}_{j+1} / \beta_{j+1}$
     \State $v_{j+1} = \hat{v}_{j+1} / \theta_{j+1}$
   \EndFor
\end{algorithmic}
\end{algorithm}

\section{Adaptive subspace extensions to increase BOROK stability}
\label{sec:BOROKext}

A potential approach to increasing numerical stability involves extending the reduced spaces with vectors not directly generated from the biorthogonalization procedure, and thus, not necessarily part of the Krylov bases. For ROK methods, this idea was shown to provide a dramatic increase in stability \cite{Tranquilli2019}. In this section, we first develop a general procedure for adding arbitrary vectors to the biorthogonal subspaces, then apply it to extend our subspaces with the linear system right-hand sides for stages 2 to $s$ of the BOROK integrator \eqref{eqn:borok-reduced-form}.

\subsection{Biorthogonal subspace extension}

We seek to introduce arbitrary vectors $\mathbf{a} = [a_1, ..., a_r] \in \R^{N\times r}$ into our biorthogonal bases, such that the augmented basis matrices $\V_{m+r}$ and $\W_{m+r}$ and the reduced space matrix $\Tb_{m+r}$  retain the biorthogonal properties of equation \eqref{eqn:biorthoprop}. We have:
\begin{align*}
  \V_{m+r} & = \left[ \V_m, \mathbf{v}_a\right], \\
  \W_{m+r} & = \left[ \W_m, \mathbf{w}_a\right], \\
  \V_{m+r}^T \W_{m+r} & = \I_{m+r},
\end{align*}
where $\mathbf{v}_a,\mathbf{w}_a \in \R^{N\times r}$ are computed such that our vectors $a_i$ are completely contained in our subspace:
\[
  \left(\I_N - \V_{m+r}\,\W_{m+r}^T\right) \mathbf{a} = \left(\I_N - \V_m \W_m^T - \mathbf{v}_a \mathbf{w}_a^T\right)\mathbf{a} = 0_{N\times r}.
\]
We can extract all the conditions which $\mathbf{v}_a,\mathbf{w}_a$  must satisfy:
\begin{subequations}
\label{eqn:extconds}
\begin{align}
  \W_m^T\, \mathbf{v}_a & = 0_{m\times r}, \label{eqn:extconds1} \\
  \V_m^T \mathbf{w}_a & = 0_{m\times r}, \\
  \mathbf{v}_a^T \mathbf{w}_a & = \I_r, \\
  \mathbf{a}^T  \mathbf{w}_a & = \mathbf{D}_a, \\
  \mathbf{v}_a \mathbf{D}_a & = \mathbf{a} - V_m W_m^T\, \mathbf{a} \label{eqn:extconds5}
\end{align}
\end{subequations}
where $\mathbf{D}_a \in \R^{r\times r}$ is a nonzero diagonal matrix.
Observe that conditions \eqref{eqn:extconds1} and \eqref{eqn:extconds5} are redundant, as left multiplying condition \eqref{eqn:extconds5} by $\W_m^T$ gives back \eqref{eqn:extconds1}:
\[
  \W_m^T\, \mathbf{v}_a \mathbf{D}_a = \W_m^T\, \mathbf{a} - \left(\W_m^T \,\V_m\right) \W_m^T\, \mathbf{a} = \W_m^T\, \mathbf{a} - \W_m^T\, \mathbf{a} = 0_{m\times r}.
\]
If we select the diagonal matrix $\mathbf{D}_a = \alpha \I_r$, then given condition \eqref{eqn:extconds5} to define $\mathbf{v}_a \mathbf{D}_a = \alpha \mathbf{v}_a$ and a selection for the arbitrary constant $\alpha$, the remaining three conditions define a set of least-squares problems to solve for each vector in $\mathbf{w}_a$:
\[
  \left[\V_m, \alpha \mathbf{v}_a, \mathbf{a}\right]^T \mathbf{w}_a = \left[\mathbf{0}_m, \alpha \I_r, \alpha \I_r\right]^T.
\]
Finally, we must compute the reduced space projection of the Jacobian, $\Tb_{m+r}$, again maintaining the biorthogonal properties from equation \eqref{eqn:biortho-tmat}:
\[
  \Tb_{m+r} = \W_{m+r}^T\, \J\, \V_{m+r}.
\]
Expanding into subblocks we have:
\begin{align*}
  \Tb_{m+r} & = \left[\begin{array}{cc}
                  \Tb_m & \mathbf{t}_v \\
                  \mathbf{t}_w^T & \mathbf{t}_{vw}
                \end{array}\right] \\[1em]
          & = \left[\W_m, \mathbf{w}_a\right]^T \J \left[\V_m, \mathbf{v}_a\right] \\[1em]
          & = \left[\begin{array}{cc}
                \W_m^T \J \V_m, & \W_m^T \J \mathbf{v}_a \\
                \mathbf{w}_a^T \J \V_m, & \mathbf{w}_a^T \J \mathbf{v}_a
              \end{array}\right].
\end{align*}
The complete procedure is summarized in Algorithm \ref{alg:extgenalgorithm}.

\begin{algorithm}[ht]
\caption{Arbitrary extension of the subspace defined by a pair of biorthogonal bases.}
\label{alg:extgenalgorithm}
\begin{algorithmic}[1]
  \State Select scalar $\alpha$.
  \State Compute the projection of $\mathbf{a}$ into the existing subspace: $\mathbf{v}_a = \frac{1}{\alpha}\left(\I_N - \V_m \W_m^T\right) \mathbf{a}$.
  \State Solve a block least-squares problem for $\mathbf{w}_a$: $\left[\V_m, \alpha \mathbf{v}_a, a\right]^T \mathbf{w}_a = \left[0_m, \alpha \I_r, \alpha \I_r\right]^T$ for $\mathbf{w}_a$.
  \State Construct augmented basis matrices:
         \begin{itemize}
           \item $\V_{m+r} = \left[\V_m, \mathbf{v}_a\right]$,
           \item $\W_{m+r} = \left[\W_m, \mathbf{w}_a\right]$.
         \end{itemize}
  \State Compute the augmentation vectors for $T_{m+r}$:
         \begin{itemize}
           \item $\mathbf{t}_v = \W_m^T\,  \J \mathbf{v}_a$,
           \item $\mathbf{t}_w = \V_m^T\,  \J^T \mathbf{w}_a$,
           \item $\mathbf{t}_{vw} = \mathbf{w}_a^T\,  \J \mathbf{v}_a$.
         \end{itemize}
  \State Construct reduced space projection matrix $\Tb_{m+r} = \left[\begin{array}{cc} \Tb_m, & \mathbf{t}_v \\ \mathbf{t}_w^T, & \mathbf{t}_{vw} \end{array}\right]$.
\end{algorithmic}
\end{algorithm}

Before we can make use of this subspace extension algorithm in the BOROK method, we need analogs to \eqref{eqn:biortho-jv} and \eqref{eqn:biortho-jtw}, the remaining two equations in the biorthogonal relations \eqref{eqn:biorthoprop}. 

\begin{lemma}[Biorthogonal properties with extended subspace.]
When the subspace extension from Algorithm \ref{alg:extgenalgorithm} is used, the biorthogonal relations \eqref{eqn:biorthoprop} take the form:
 \begin{subequations}
  \label{eqn:biorthopropext}
  \begin{align}
   \V_{m+r}^T \W_{m+r} & = \I_{m+r}, \label{eqn:bopropext_orth} \\
   \Tb_{m+r} & = \W_{m+r}  \J \V_{m+r}, \label{eqn:bopropext_t} \\
   \J \V_{m+r} & = \V_{m+r} \Tb_{m+r} + \left(\I_N - \V_{m+r} \W_{m+r}^T\right) \J \left(v_m e_m^T + \mathbf{v}_a\sum_{i=1}^r e_{m+i}^T\right), \label{eqn:bopropext_jv} \\
   \J^T \W_{m+r} & = \W_{m+r} \Tb_{m+r}^T + \left(\I_N - \W_{m+r} \V_{m+r}^T\right) \J^T \left(w_m e_m^T + \mathbf{w}_a\sum_{i=1}^r e_{m+i}^T\right). \label{eqn:bopropext_jtw}
  \end{align}
 \end{subequations}
\end{lemma}
\begin{proof}
 From Algorithm \ref{alg:extgenalgorithm}, we have both properties \eqref{eqn:bopropext_orth} and \eqref{eqn:bopropext_t} by construction. For property \eqref{eqn:bopropext_jv}, we expand using the original biorthogonal properties:
 \begin{align*}
  \J \V_{m+r} - \V_{m+r} \Tb_{m+r} & = \left[\V_m \Tb_m + \left(\I_N - \V_m \W_m^T\right) \J v_m e_m^T, \J \mathbf{v}_a\right] - \left[\V_m \Tb_m + \mathbf{v}_a \mathbf{w}_a^T \J v_m e_m^T, \V_{m+r} \W_{m+r}^T\,  \J \mathbf{v}_a\right] \\
  & = \left[\left(\I_N - \V_m \W_m^T - \mathbf{v}_a \mathbf{w}_a^T\right) \J v_m e_m^T, \left(\I_N - \V_{m+r} \W_{m+r}^T\right) \J \mathbf{v}_a\right].
 \end{align*}
 Substituting $\V_{m+r} \W_{m+r}^T = \V_m \W_m^T + \mathbf{v}_a \mathbf{w}_a^T$, and making use of canonical basis vectors $e_i \in \R^n$, completes the proof for property \eqref{eqn:bopropext_jv}. Similarly, for property \eqref{eqn:bopropext_jtw} we have:
 \begin{align*}
  \J^T \W_{m+r} - \W_{m+r} \Tb_{m+r}^T & = \left[\W_m \Tb_m^T + \left(\I_N - \W_m \V_m^T\right) \J^T w_m e_m^T, \J^T \mathbf{w}_a\right] \\
  & \quad - \left[\W_m \Tb_m^T + \mathbf{w}_a \mathbf{v}_a^T \J^T w_m e_m^T, \W_{m+r} \V_{m+r}^T\,  \J^T \mathbf{w}_a\right] \\
  & = \left[\left(\I_N - \W_m \V_m^T - \mathbf{w}_a \mathbf{v}_a^T\right) \J^T w_m e_m^T, \left(\I_N - \W_{m+r} \V_{m+r}^T\right) \J^T \mathbf{w}_a\right].
 \end{align*}
 Again, rewriting with canonical basis vectors completes the proof.
 \qed
\end{proof}

With the above general procedure for extending  biorthogonal bases at hand, we now  have a variety of choices for vectors to add. One possibility is to consider basis recycling \cite{Parks2006,Ye2008}, where we save some basis vectors from previous timesteps, allowing us to minimize the number of new basis vectors computed at each step. This basis extension procedure also allows us to implement type-2 ROK methods \cite[Section 3.2]{Sandu_2014_ROK}, which make use of higher order derivative information from the ODE derivative function to satisfy the additional order conditions for ROS methods when the order is greater than the subspace dimension $m$. A very simple possibility comes from our examination of the BOROK stage residual in Theorem \ref{thm:stgresidual}: the most significant contributors to the residual that are not naturally contained in the Krylov subspaces are the ODE right hand side function evaluations $F_j$ from each stage.

\subsection{Stage residuals with subspace extension}
\label{subsec:BOROKext}

Now we consider a particular application of Algorithm \ref{alg:extgenalgorithm} to the BOROK method. From Theorem \ref{thm:stgresidual} we see that, for stages 2 to $s$, the residual contains the terms $\left(F_j - V_m\,\psi_j\right)$ that represent those components of the ODE right hand side that are not captured by the Krylov subspaces. Thus, if we augment the Krylov basis matrices to include the $F_j$, we expect to produce smaller residuals at these stages, and potentially improve the total method stability.

\begin{definition}[Basis matrices with subspace extension]
 At the \emph{i}th stage of a BOROK method, we define the basis matrices $\V_{m;i}$, $\W_{m;i}$ and the reduced space matrix $\Tb_{m;i}$ as, at $i = 1$:
 \begin{align}
  \V_{m;1} & = \V_m, \\
  \W_{m;1} & = \W_m, \\
  \Tb_{m;1} & = \Tb_m.
 \end{align}
 And, for stages $2$--$s$, $\V_{m;i}$, $\W_{m;i}$ and $\Tb_{m;i}$ are defined by extending the previous stage's matrices $\V_{m;i-1}$, $\W_{m;i-1}$, and $\Tb_{m;i-1}$ with $F_i$ from \eqref{eqn:redstagef} via algorithm \ref{alg:extgenalgorithm}.
\end{definition}

\begin{definition}[BOROK reduced form stage equations with subspace extension]
 We define the BOROK reduced form stage equations in terms of these extended matrices as
 \begin{subequations}
 \label{eqn:redstageeqext}
  \begin{align}
   \psi_i & = \W_{m;i}^T F_i, \label{eqn:redstageeqext1} \\
   \left(\I_m - h\gamma\,\Tb_{m;i}\right) \widehat{\lambda}_i & = h\, \psi_i + h \Tb_{m;i} \sum_{j=1}^{i-1}\gamma_{i,j}\,\widehat{\lambda}_j, \label{eqn:redstageeqext2} \\
   k_j & = \V_{m;i}\,\widehat{\lambda}_j + h\left(F_j - \V_{m;i}\,\psi_j\right). \label{eqn:redstageeqext3}
  \end{align}
 \end{subequations}
\end{definition}

\begin{theorem}[Stage linear system residual of a BOROK method with subspace extension]
The residual of the \emph{i}th stage linear system of an \emph{s}-stage biorthogonal ROK method with subspace extension is:
 \begin{equation}
  r_{m;i} = -h\left(\I_N - \V_{m;i}\,\W_{m;i}^T\right) \J \left(v_m e_m^T + \sum_{k=1}^{i-1}v_{m+k}e_{m+k}^T\right) \sum_{j=1}^i \gamma_{i,j}\,\widehat{\lambda}_j,
 \end{equation}
 where $\widehat{\lambda}_j \in \R^{m+i}$ is defined by
 \begin{equation}
  \widehat{\lambda}_j = \left\{\begin{array}{cc} \left[\lambda_j^T, 0, ..., 0\right]^T, & j < i, \\ \lambda_j, & j = i. \end{array} \right.
 \end{equation}
\end{theorem}

\begin{proof}
 Following the proof of Theorem \ref{thm:stgresidual}, we start with the full space stage equation \eqref{eqn:ros-stage-eq} with a residual term $r_{m;i}$:
 \begin{equation*}
  \left(\I_N - h\gamma\,\J\right) k_i = h F_i + h \J \sum_{j=1}^{i-1} \gamma_{i,j} k_j + r_{m;i},
 \end{equation*}
 From the reduced space form \eqref{eqn:redstageeqext}, substituting \eqref{eqn:redstageeqext1} into \eqref{eqn:redstageeqext3} results in the term $F_j - \V_{m;i}\,\W_{m;i}^T F_j = 0$ by the extension of the subspace with $F_j$, so equation \eqref{eqn:redstageeqext3} becomes:
 \[
  k_j = \V_{m;i}\,\widehat{\lambda}_j.
 \]
 Now, we can isolate the residual term in the full space form and begin substituting in the reduced space equations:
 \begin{align*}
  r_{m;i} & = k_i - h F_i - h \J \sum_{j=1}^i \gamma_{i,j} k_j \\
          & = \V_{m;i} \widehat{\lambda}_i - h \V_{m;i}\,\psi_i - h\,\J \V_{m;i} \sum_{j=1}^i \gamma_{i,j} \widehat{\lambda}_j \\
          & = \V_{m;i}\left(\widehat{\lambda}_i - h \psi_i\right) -h \left(\V_{m;i} \Tb_{m;i} + \left(\I_N - \V_{m;i}\,\W_{m;i}^T\right)\J\left(v_m e_m^T + \sum_{k=1}^{i-1}v_{m+k}e_{m+k}^T\right)\right) \sum_{j=1}^i \gamma_{i,j} \widehat{\lambda}_j \\
          & = \V_{m;i}\underbrace{\left(\widehat{\lambda}_i - h\psi_i - h \Tb_{m;i}\sum_{j=1}^i \gamma_{i,j}\,\widehat{\lambda}_j\right)}_{= 0} - \left(\I_N - h\V_{m;i}\,\W_{m;i}^T\right)\J \left(v_m e_m^T + \sum_{k=1}^{i-1} v_{m+k} e_{m+k}^T\right) \sum_{j=1}^i \gamma_{i,j} \widehat{\lambda}_j.
 \end{align*}
 \qed
\end{proof}

\section{Numerical results}
\label{sec:results}

In this section we test the several implementations of the new BOROK methods on a set of problems. As the base ROK methods (and other K-methods) have been previously compared against other standard methods (see \cite{Sandu_2014_ROK, Sandu_2014_expK, Sandu_2017_analytical-JacVec, Sandu_2019_EPIRKW, Tranquilli2019}), we restrict ourselves to comparing BOROK only against the base ROK methods and compatible variations.

The experiments in this section are performed in Matlab, with errors evaluated against reference solutions computed with Matlab built-in \texttt{ode15s} time integrator set to the tightest available tolerance of $100\times \text{eps}$.

All the figures in this section use the same labeling scheme, as follows. First, datasets prefixed by an $L$ denote results from BOROK methods (making use of Lanczos biorthogonalization). Datasets labeled $M$ are results obtained using the stated fixed number of basis vectors, where those labeled $R$ are from an adaptive number of basis vectors, selected based on the stated residual tolerance. Subspace extension (as in section \ref{sec:BOROKext}) is indicated in the label by the suffix $ext$ and can apply to either fixed or adaptive basis size selection. Also, lines labeled $R = \text{tol}$ are experiments where the Arnoldi residual tolerance equal to the tolerance for the adaptive stepsize error controller.

\subsection{Nonstiff test problem: shallow water equations}

First we examine the convergence and relative performance of the methods on a simple nonstiff PDE, the shallow water equations in Cartesian coordinates \cite{Liska97compositeschemes}. The system of equations is:
\begin{align}
\label{eqn:shallowwater}
 \displaystyle\frac{\partial}{\partial t} h + \frac{\partial}{\partial x} (uh) + \frac{\partial}{\partial y} (vh) &= 0, \\[1em]
 \displaystyle\frac{\partial}{\partial t} (uh) + \frac{\partial}{\partial x} \left(u^2 h + \frac{1}{2} g h^2\right) + \frac{\partial}{\partial y} (u v h) &= 0,  \\[1em]
 \displaystyle\frac{\partial}{\partial t} (vh) + \frac{\partial}{\partial x} (u v h) + \frac{\partial}{\partial y} \left(v^2 h + \frac{1}{2} g h^2\right) &= 0, 
\end{align}
with reflective boundary conditions, where $u(x,y,t)$, $v(x,y,t)$ are the flow velocity components, $h(x,y,t)$ is the fluid height, and $g$ is the gravitational acceleration constant. The spatial discretization's right-hand side is implemented using centered finite differences on an $64 \times 64$ grid, and the system \eqref{eqn:shallowwater} is brought to the standard ODE form \eqref{eqn:ivp} with
\begin{equation*}
y = \left[ u \, \, v \, \, h\right]^T \in \R^{N}, \quad f_y(t,y) = \J \in \R^{N \times N}, ~~ N = 3 \times 64 \times 64, ~~ t \in \left[0, 5\right].
\end{equation*}
We call this problem nonstiff, as the largest negative real eigenvalue of the initial Jacobian is less than $12$.

\begin{figure}
  \centering
  \includegraphics[width=4in]{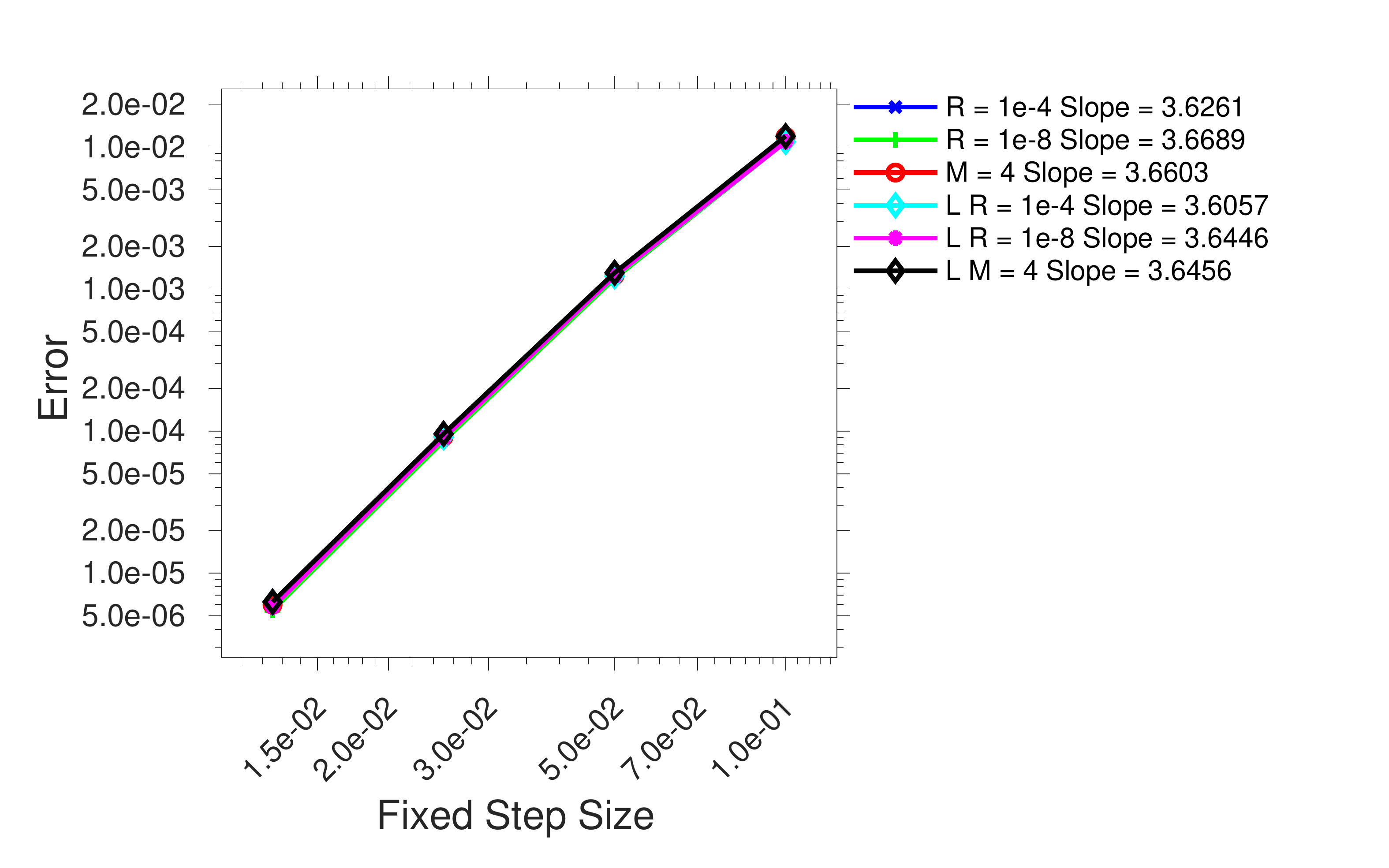}
  \caption{Order of convergence for fixed stepsize ROK and BOROK methods tested on the shallow water equations \eqref{eqn:shallowwater}. All six tested method configurations give nearly identical results, and consequently there is extensive overlap between plot lines.}
  \label{fig:swe_cart_conv}
\end{figure}

Figure \ref{fig:swe_cart_conv} shows the fixed stepsize convergence results. BOROK demonstrates identical convergence behavior to the base ROK methods, with all methods showing their theoretical fourth order convergence rate.

\begin{figure}
\centering
\subfigure[Work-precision diagram.]{
  \includegraphics[width=4in]{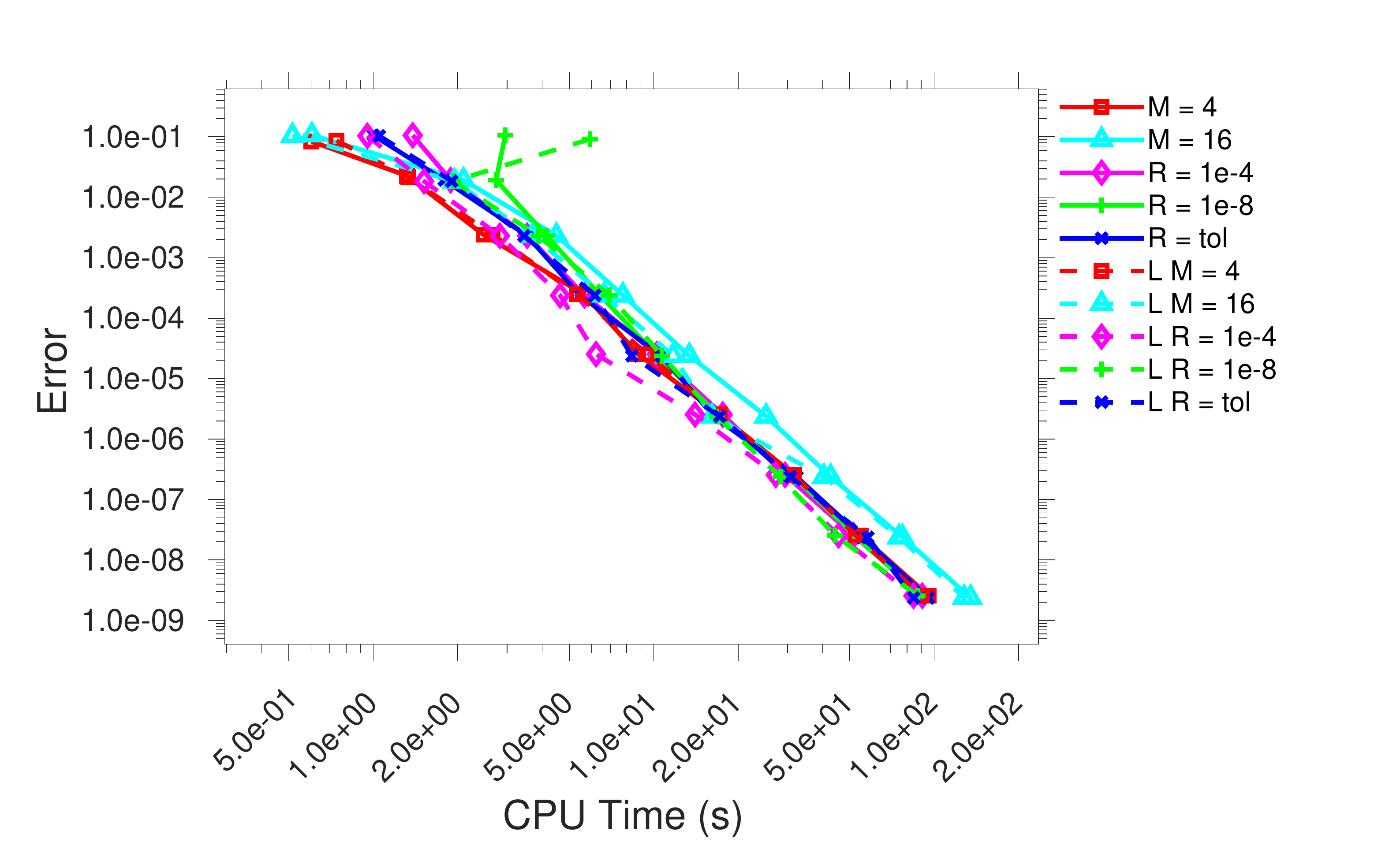}
  \label{fig:swe_cart_times}
} \\
\subfigure[Timesteps to solution.]{
  \includegraphics[width=4in]{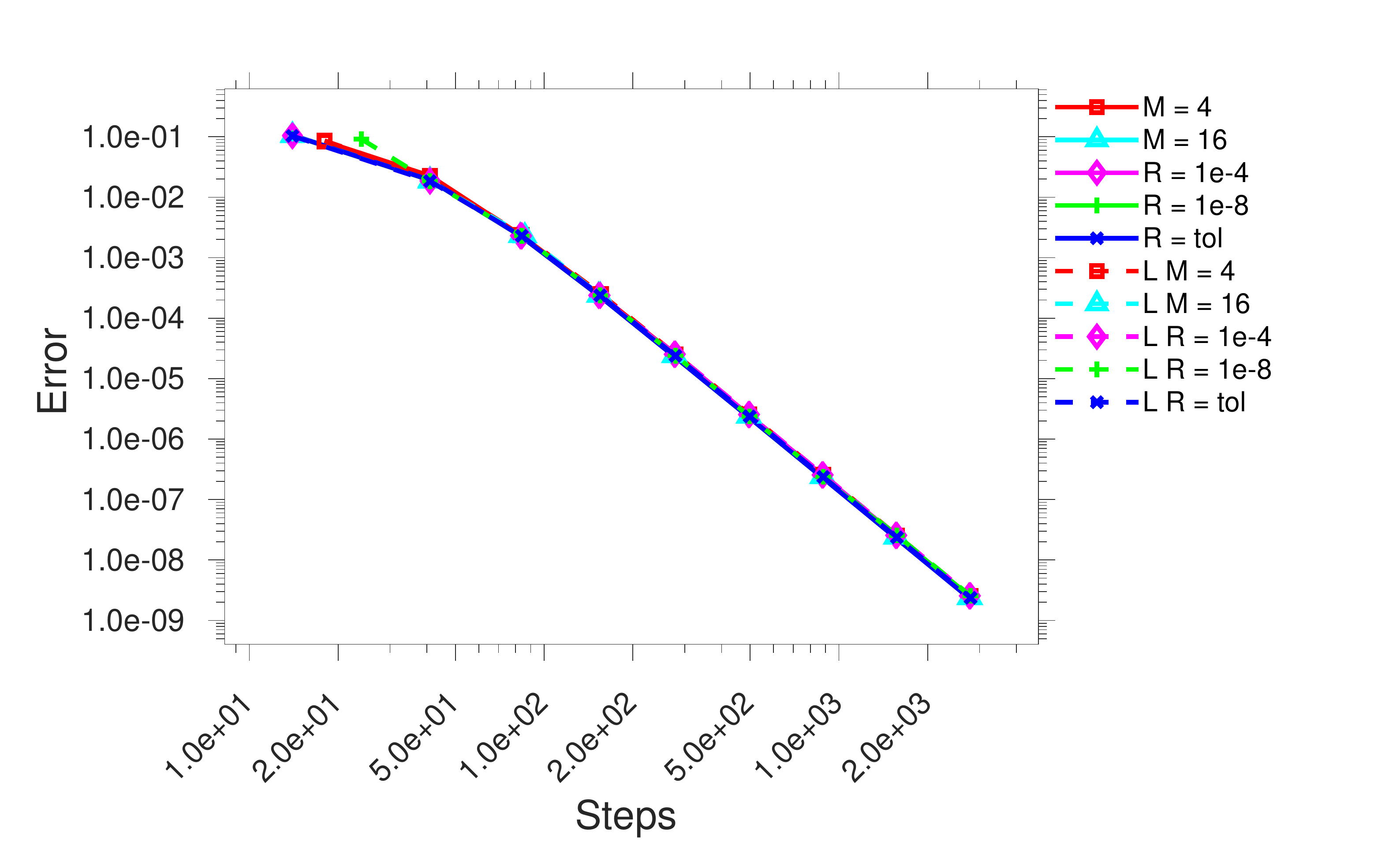}
  \label{fig:swe_cart_steps}
}
\caption{Adaptive step results for the shallow water equations in rectangular coordinates \eqref{eqn:shallowwater}. Solid lines represent ROK methods and dashed lines represent the new BOROK methods. Extensive overlap occurs because of the nonstiff nature of the problem.}
\label{fig:swe_cart}
\end{figure}
Figure \ref{fig:swe_cart} shows the solution accuracy versus the compute time and step count for adaptive stepsize BOROK and ROK implementations tested over a range of tolerances. Figure \ref{fig:swe_cart_times} gives timing results, and Figure \ref{fig:swe_cart_steps} shows the corresponding number of steps. The shallow water equations are nonstiff, so the integrators experience no meaningful stability restrictions, and the number of steps is limited only by the desired tolerance values. Many plots in Figure \ref{fig:swe_cart_steps} fall on top of each other, meaning that all methods have similar accuracy. Thus, the small timing differences observed in Figure \ref{fig:swe_cart_times} are the result of per-timestep costs, with fixed basis size methods with larger $M = 16$ proving the slowest, and all other methods demonstrating very little variation. The only exceptions occur at very loose tolerances where the interaction between the stepsize controller and adaptive basis size methods results in additional overhead. This test effectively demonstrates that there is relatively little difference in the performance of between BOROK and ROK methods for nonstiff problems.

\subsection{Stiff test problem: Gray-Scott reaction-diffusion}

Next, we consider the Gray-Scott reaction-diffusion model arising from a two species system involved in two chemical reaction with retirement as described in \cite{Gray1983, Gray1984}:
\begin{align*}
U + 2V & \rightarrow 3V, \\
V & \rightarrow P.
\end{align*}
With spatial diffusion the full model is described by the following PDE:
\begin{equation}
\label{eqn:grayscott}
\begin{split}
\displaystyle\frac{\partial u}{\partial t} & = \varepsilon_1 \Delta u - u v^2 + F (1 - u), \\[1em]
\displaystyle\frac{\partial v}{\partial t} & = \varepsilon_2 \Delta v + u v^2 - (F + k) v,
\end{split}
\end{equation}
where $\varepsilon_1$ and $\varepsilon_2$ are diffusion rates and $F$ and $k$ are reaction rates. An implementation of this model is part of the ODE Test Problems suite \cite{roberts2019otp}. It employs a second order finite difference spatial discretization on a uniform $128 \times 128$ 2D grid with periodic boundary conditions. As with the shallow water equations above, the system \eqref{eqn:grayscott} is brought to ODE form \eqref{eqn:ivp} with
\[
y = \left[u \,\, v\right]^T \in \R^N, \quad f_y(t,y) = \J \in \R^{N \times N}, \quad N = 2 \times 128 \times 128, \quad t \in \left[0, 2\right].
\]
For faster diffusion or reaction rates, the Gray-Scott model becomes very stiff (in the tested form, $\varepsilon_1 = 0.2$, $\varepsilon_2 = 0.1$, $F = 0.04$, $k = 0.06$, and the largest negative eigenvalue of the initial Jacobian is $-4.2\times 10^3$); this allows us to test the differences between the ROK method's Arnoldi iteration with an $M$-term recurrence, and BOROK's Lanczos biorthogonalization with only a two-term recurrence.

\begin{figure}
\centering
\subfigure[Work-precision diagram.]{
  \includegraphics[width=4in]{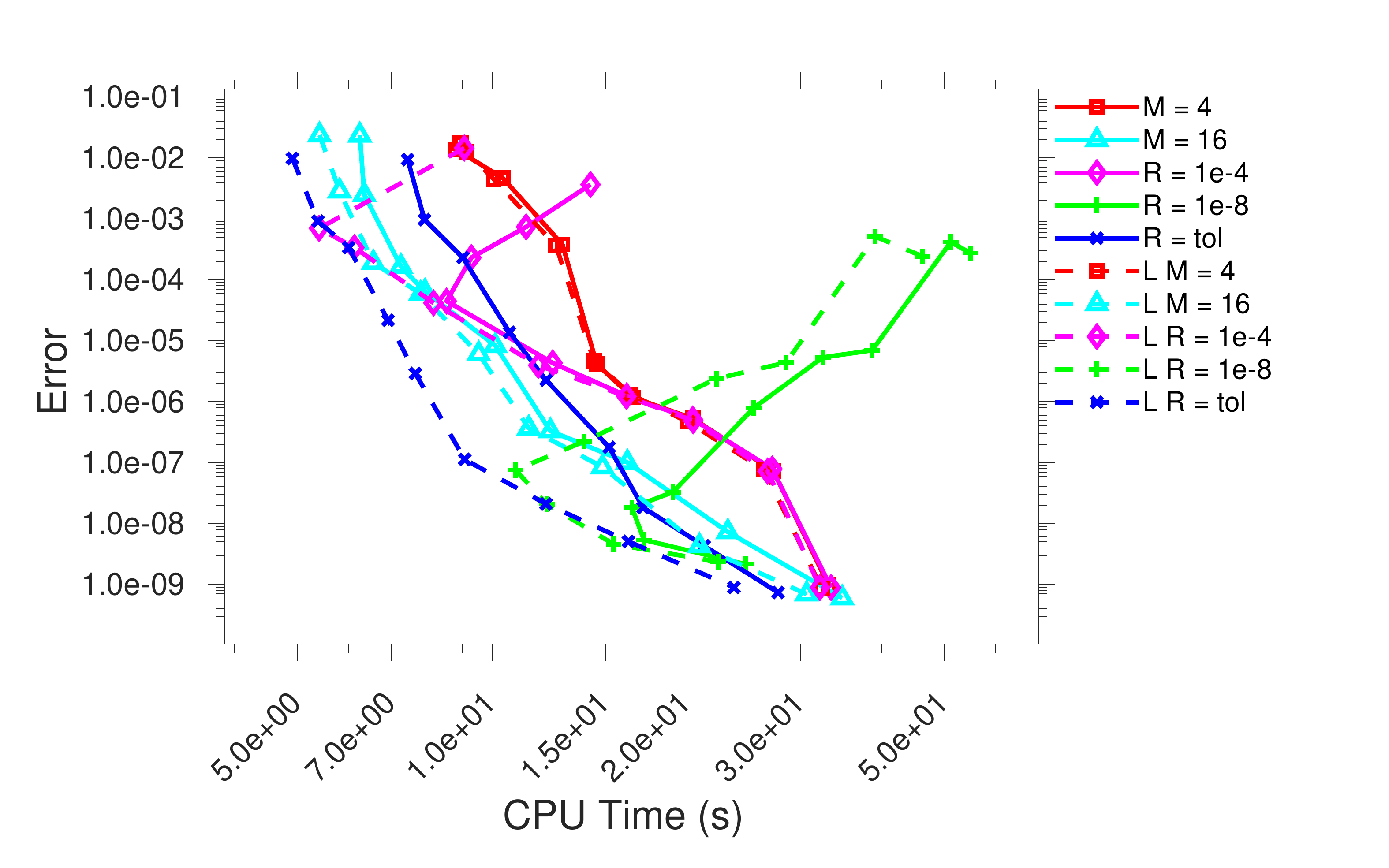}
  \label{fig:gray_scott_times}
} \\
\subfigure[Timesteps to solution.]{
  \includegraphics[width=4in]{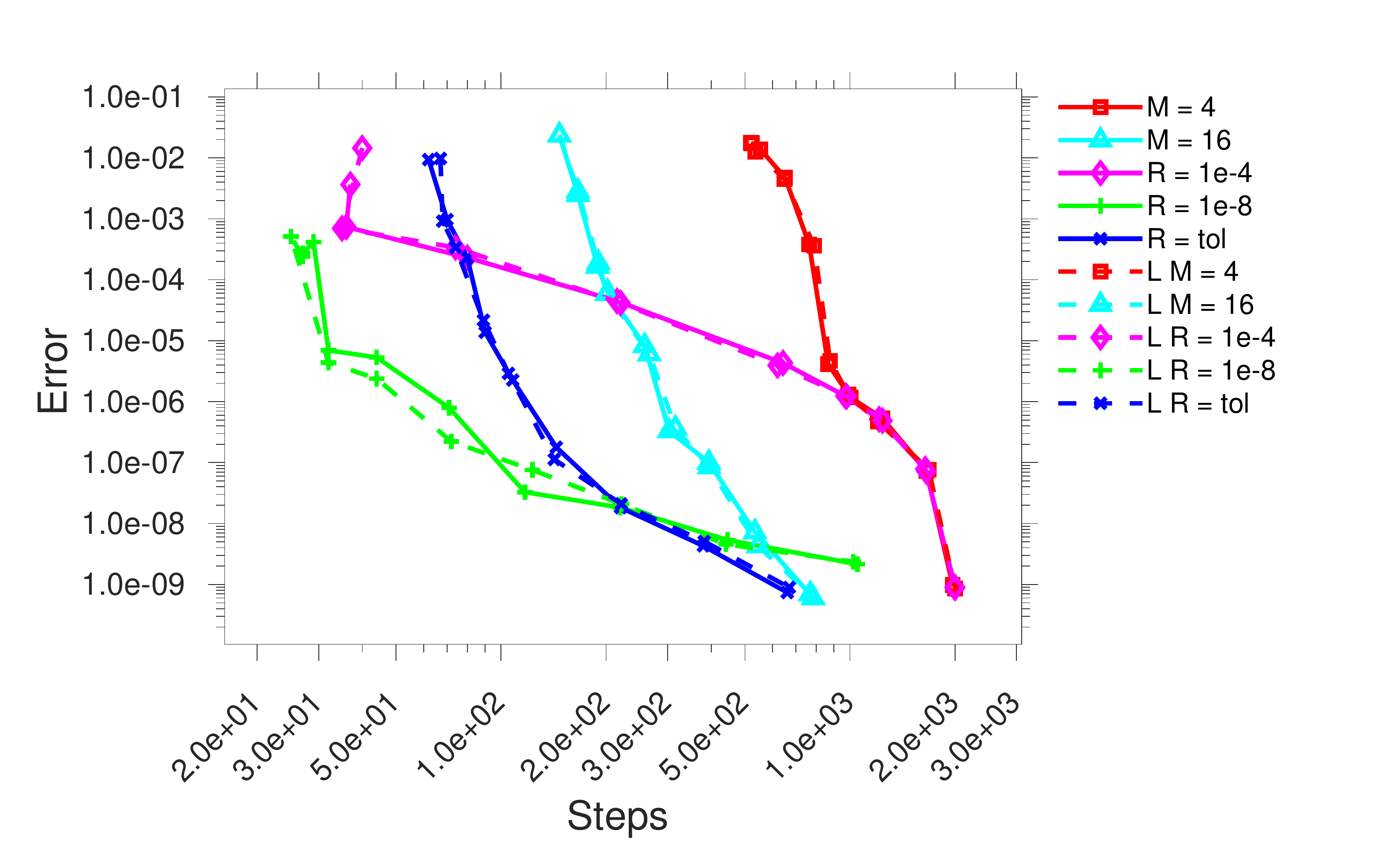}
  \label{fig:gray_scott_steps}
}
\caption{Results for the Gray-Scott reaction-diffusion problem \eqref{eqn:grayscott}. Solid lines represent ROK methods and dashed lines represent the new BOROK methods.}
\label{fig:gray_scott}
\end{figure}

Figure \ref{fig:gray_scott} shows results for fixed and adaptive base implementations of ROK and BOROK methods applied to the Gray-Scott model with a uniform step size. The results in Figure \ref{fig:gray_scott_steps} show that ROK and BOROK methods which share basis size configurations have very similar step counts: the fixed basis size configurations match almost exactly, and only small variations occur in the adaptive configurations. We conclude that replacing ROK's Arnoldi iteration with BOROK's biorthogonalization does not change the numerical stability of the time integration methods. Fixed basis size method results are nearly vertical lines  in the step counts Figure \ref{fig:gray_scott_steps}. This indicates extensive stability-related stepsize restrictions, resulting in large step counts which change very little as the required tolerance is tightened. 

The adaptive basis size configurations also demonstrate interesting stability behavior: for set residual tolerance methods ($R = 1e-4$ and $R = 1e-8$), we see regions where no stability-related stepsize restriction is present (usually centered around the method tolerances that match the set residual tolerance). So, the $R = 1e-4$ configuration sees no stepsize restriction for method tolerances in the range of $10^{-2}$ to $10^{-4}$, but for tighter method tolerances this configurations falls back to the behavior of the $M = 4$ fixed basis size configuration. The $R = 1e-8$ adaptive basis size configuration encounters the least stepsize restriction of all the tested configurations, however looking at the timing results in Figure \ref{fig:gray_scott_times} shows the downside: by over-resolving the linear system solutions, these configurations take the least number of steps but take the longest time to solution overall. Thus, by combining the best results from each adaptive basis size configuration, we justify the $R = \text{tol}$ configurations, which match the residual tolerance to the method tolerance, resulting in a medium number of steps taken and the most consistently good timing performance.

Timing results reported in Figure \ref{fig:gray_scott_times} illustrate the most significant benefit of replacing the  Arnoldi iteration with Lanczos biorthogonalization: due to the large basis sizes needed for good stability on this stiff problem, for every basis configuration (aside from $M = 4$) we see a significant decrease in time to solution for the BOROK methods over the ROK methods. For the fixed basis size $M = 16$ configuration, the timing gap remains largely constant across all data points, but for adaptive basis sizes, e.g., the $R = 1e-4$ configuration, the time gap varies with the expected size of the basis needed to satisfy the tolerance.

\begin{remark}
For adaptive basis size configurations, as the required tolerance tightens and stepsizes become restricted by accuracy considerations, one expects that a given residual tolerance can be satisfied with fewer basis vectors. Theorem \ref{thm:stgresidual} shows that the residuals scale with the stepsize $h$, so a smaller $h$ may balance larger contributions to the residual due to fewer vectors in the basis.
\end{remark}

\begin{figure}
\centering
\subfigure[Work-precision diagram.]{
  \includegraphics[width=4in]{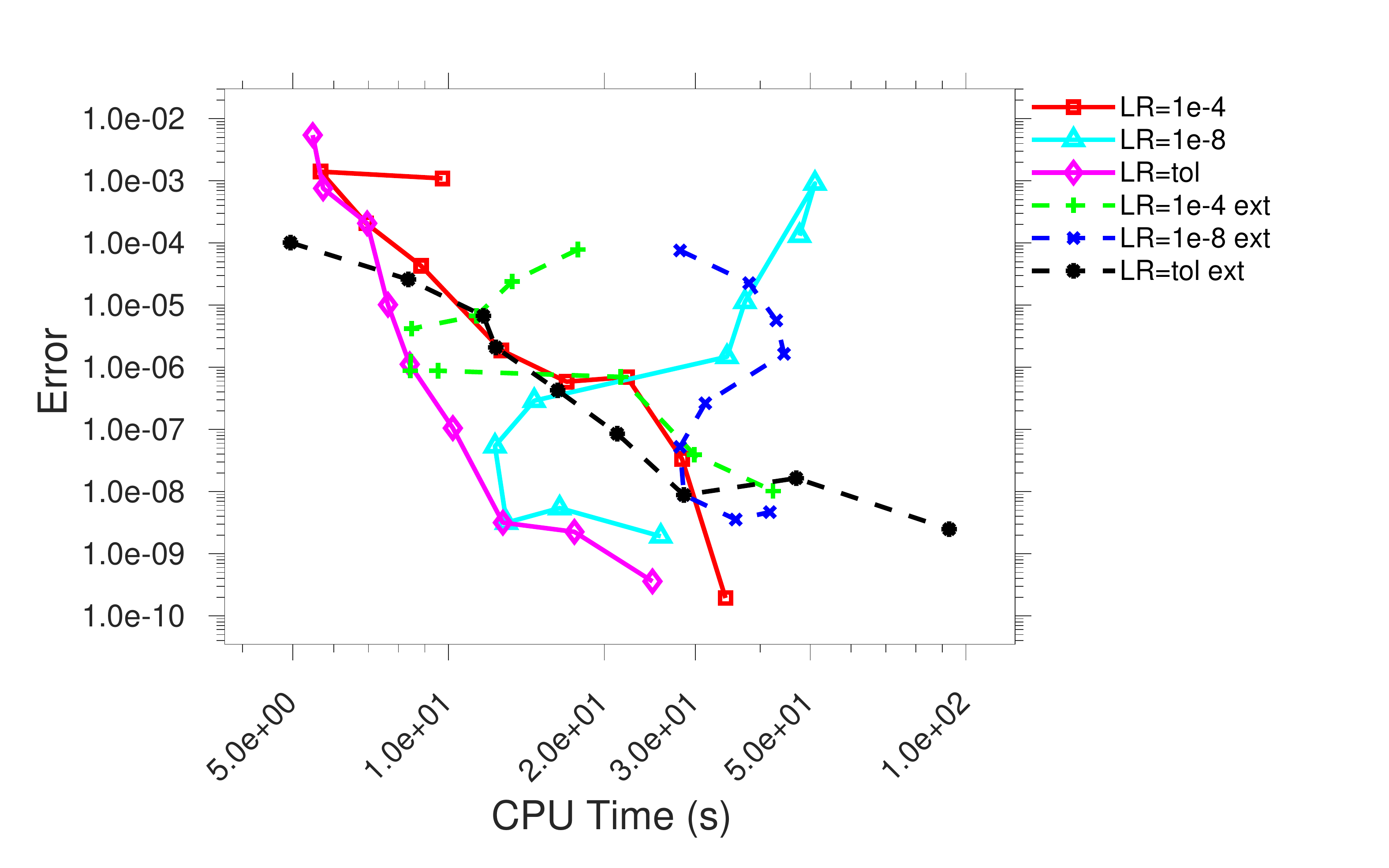}
  \label{fig:gray_scott_ext_times}
} \\
\subfigure[Timesteps to solution.]{
  \includegraphics[width=4in]{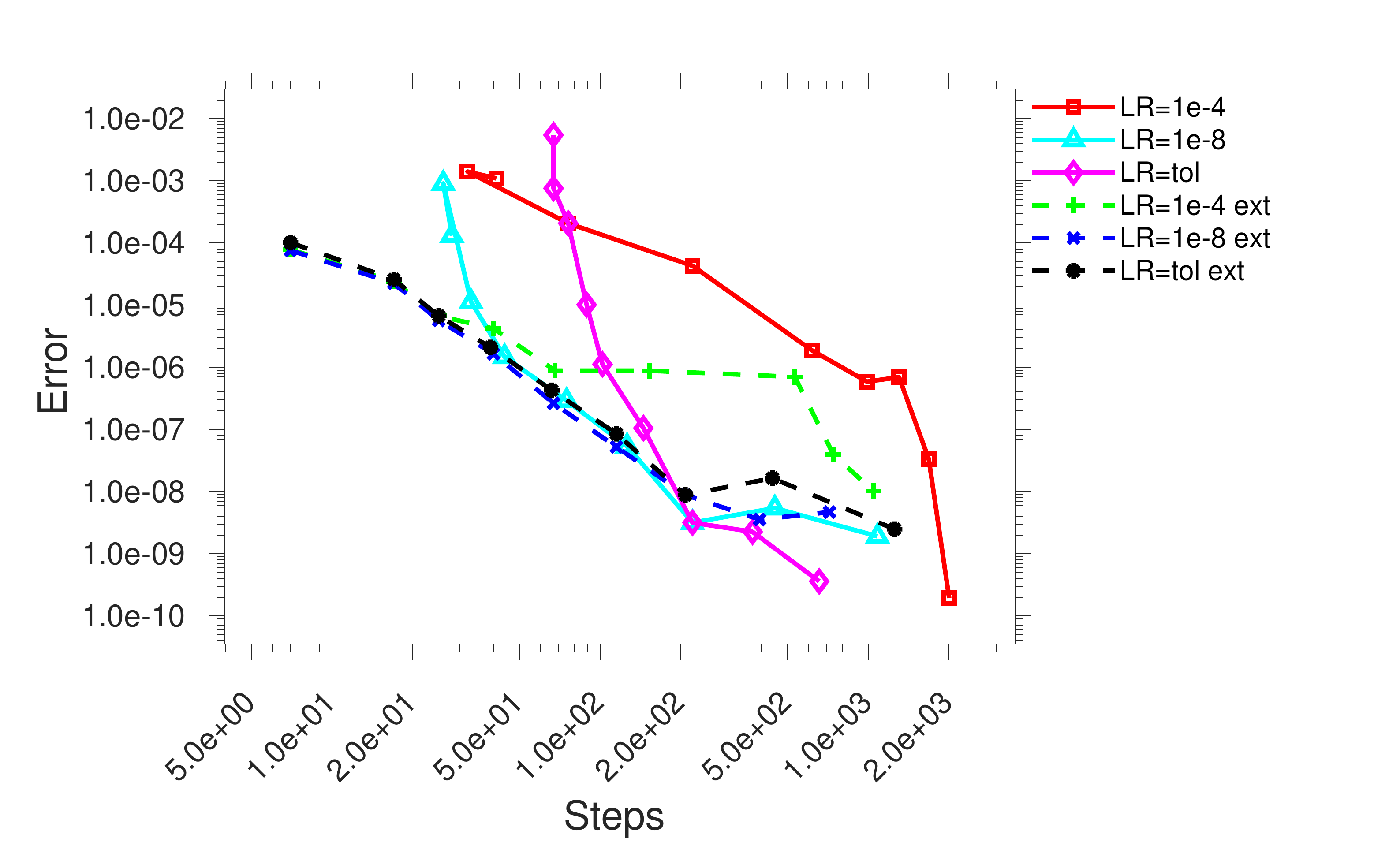}
  \label{fig:gray_scott_ext_steps}
}
\caption{Extended subspace results for the Gray-Scott reaction-diffusion problem \eqref{eqn:grayscott}. Solid lines represent BOROK methods and dashed lines represent the BOROK methods with subspace extension.}
\label{fig:gray_scott_ext}
\end{figure}

Figure \ref{fig:gray_scott_ext} shows results for BOROK applied to the Gray-Scott model with and without subspace extension (as described in Section \ref{sec:BOROKext}). The step counts in Figure \ref{fig:gray_scott_ext_steps} show that the extended basis configurations do improve the stability of the BOROK method, with each corresponding extended basis configuration requiring fewer steps to obtain similar error levels. We notice that the $LR = 1e-8 \text{ ext}$ and $LR = \text{tol ext}$ configurations appear to have their stepsizes bound only by accuracy considerations. However, Figure \ref{fig:gray_scott_ext_times} shows the timing cost of these stability improvements. Because the basis extension requires the solution to a least-squares problem of size $N\times(M+1)$ at each stage, adding basis extension to the method is very costly, and the moderate reduction in step count is insufficient to make up for the additional cost on this model. The Gray-Scott model has a very sparse Jacobian matrix, so matrix-vector products with the Jacobian (or it's transpose) are cheap to compute, making the addition of vectors to the Krylov space inexpensive. Thus, for this model, it is more efficient to add basis vectors via the Lanczos biorthogonalization procedure than it is to extend the basis with external vectors from outside the subspace.

\subsection{Stiff test problem: quasi-geostrophic model}

The 1.5 layer quasi-geostrophic (QG) model from \cite{Sakov2007} provides a simplified representation of ocean dynamics via the equations:
\begin{align}
\label{eqn:qgmodel}
  \displaystyle\frac{\partial q}{\partial t} & = -\psi_x - \varepsilon J\left(\psi,q\right) - A\Delta^3\psi + 2\pi \text{sin}\left(2\pi y\right), \\
  q & = \Delta \psi - F\psi,
\end{align}
where $J\left(\psi,q\right) = \psi_x q_y - \psi_y q_x$, and $F = 1600$, $\varepsilon = 10^{-5}$ and $A = 10^{-5}$ are constants. The tested implementation for the QG model comes from the ODE Test Problems suite \cite{roberts2019otp} and has been used previously in \cite{Popov2019}. This implementation discretizes the system \eqref{eqn:qgmodel} in terms of the stream function $\psi$, on the spatial domain $(x,y) \in [0,1]^2$, using second order central finite differences and homogeneous Dirichlet boundary conditions on a $255 \times 255$ grid. Integration is performed over a time span of $t \in \left[0, 0.01\right]$. This test problem is stiff, as the largest eigenvalue of the initial Jacobian in absolute value is approximately $-2.7 \times 10^6$. The QG model requires the solution to a Helmholtz equation, which results in a portion of the Jacobian being dense. This greatly increases the cost of Jacobian-vector products, making it a good candidate for testing the performance of BOROK schemes with basis extension.

\begin{figure}
\centering
\subfigure[Work-precision diagram.]{
  \includegraphics[width=4in]{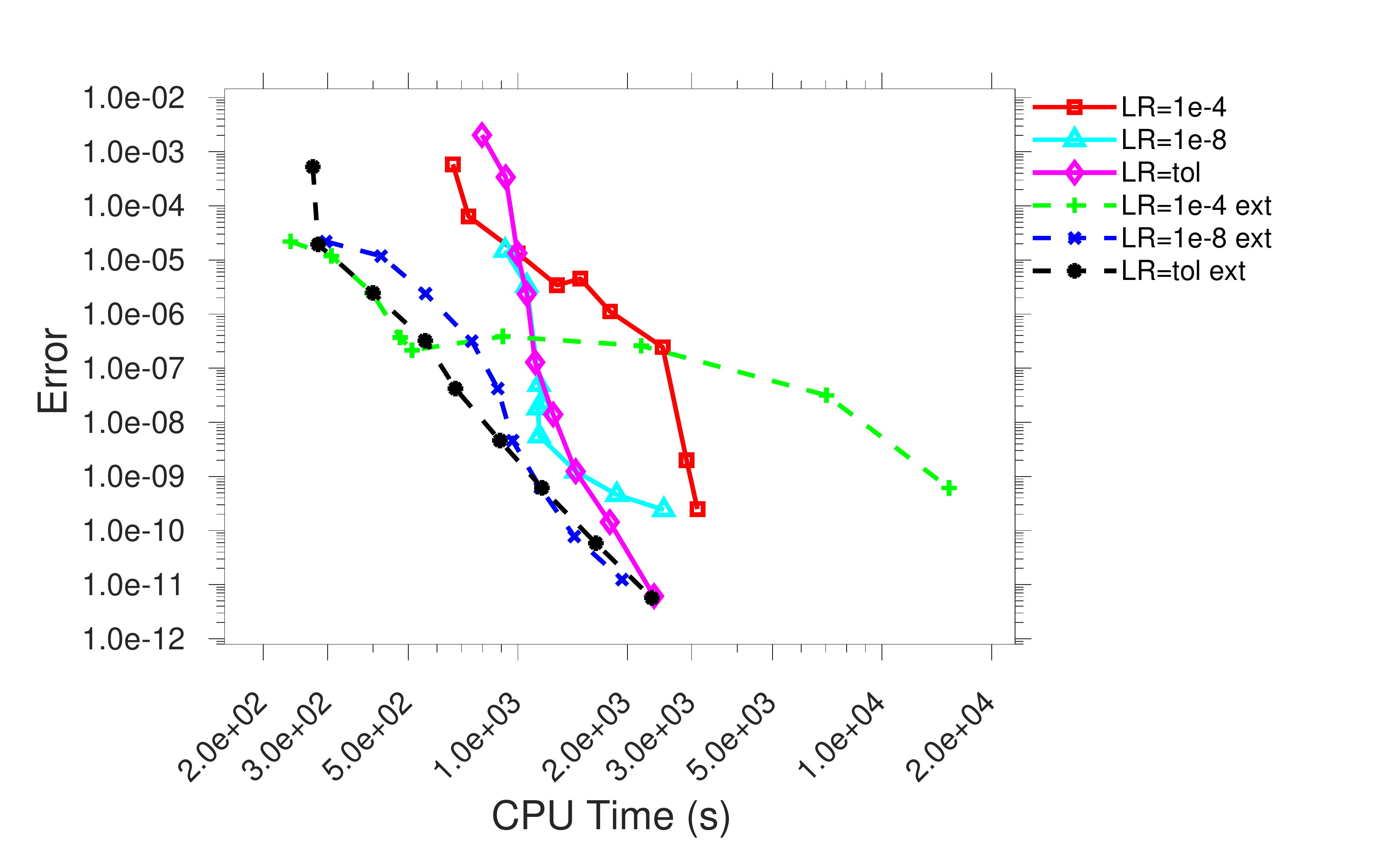}
  \label{fig:qg_ext_times}
} \\
\subfigure[Timesteps to solution.]{
  \includegraphics[width=4in]{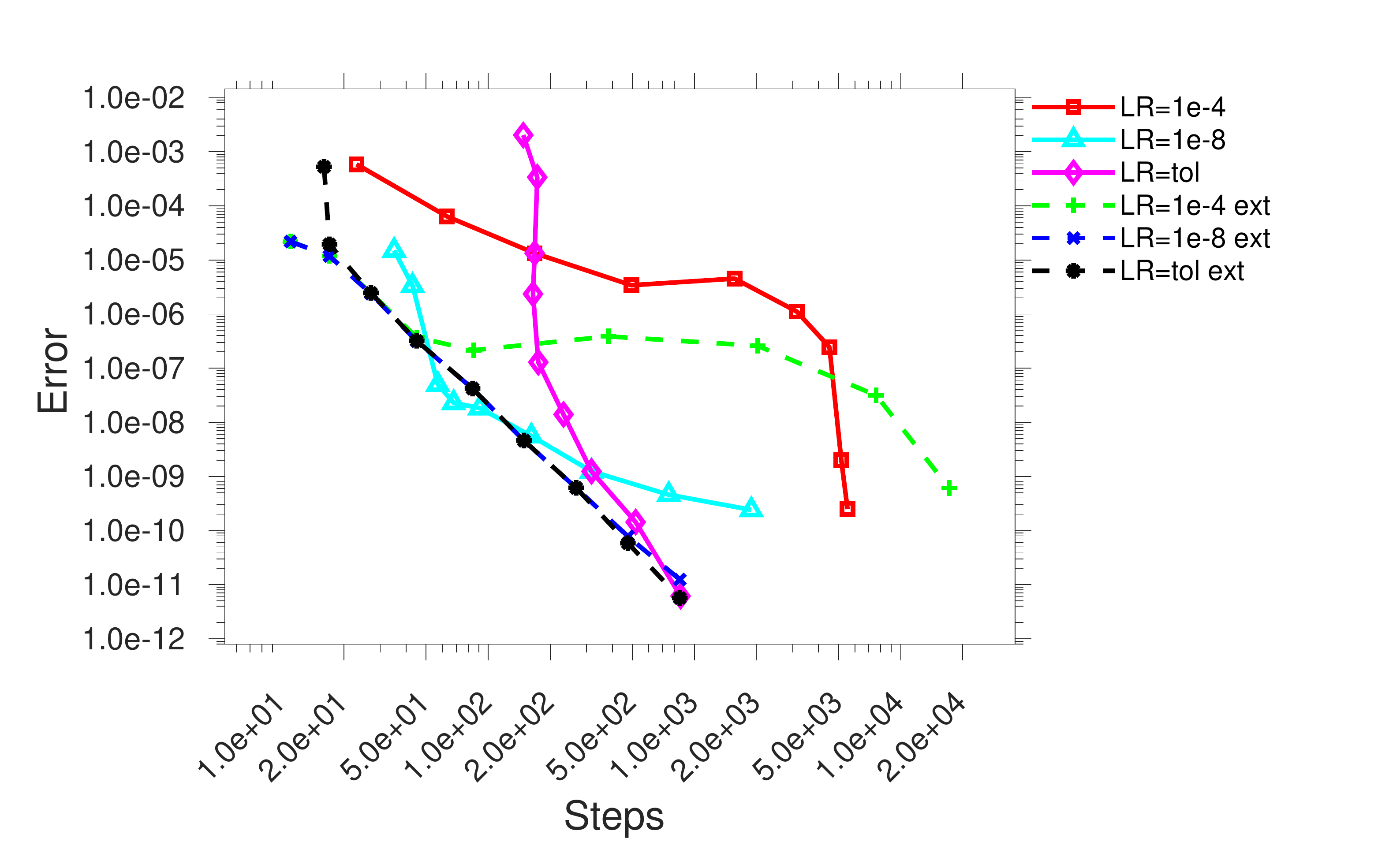}
  \label{fig:qg_ext_steps}
}
\caption{Extended subspace with adaptive basis size results for the quasi-geostrophic problem \eqref{eqn:qgmodel}. Solid lines represent BOROK methods and dashed lines represent the BOROK methods with subspace extension.}
\label{fig:qg_ext}
\end{figure}

Figure \ref{fig:qg_ext} shows results for BOROK with and without basis extension applied to the QG model. Figure \ref{fig:qg_ext_steps} illustrates the stability improvement acquired by extending the bases. The corresponding performance improvement is illustrated in Figure \ref{fig:qg_ext_times}, with the $LR = \text{tol} ext$ showing especially good performance relative to the unextended basis configuration.

\begin{figure}
\centering
\subfigure[Work-precision diagram.]{
  \includegraphics[width=4in]{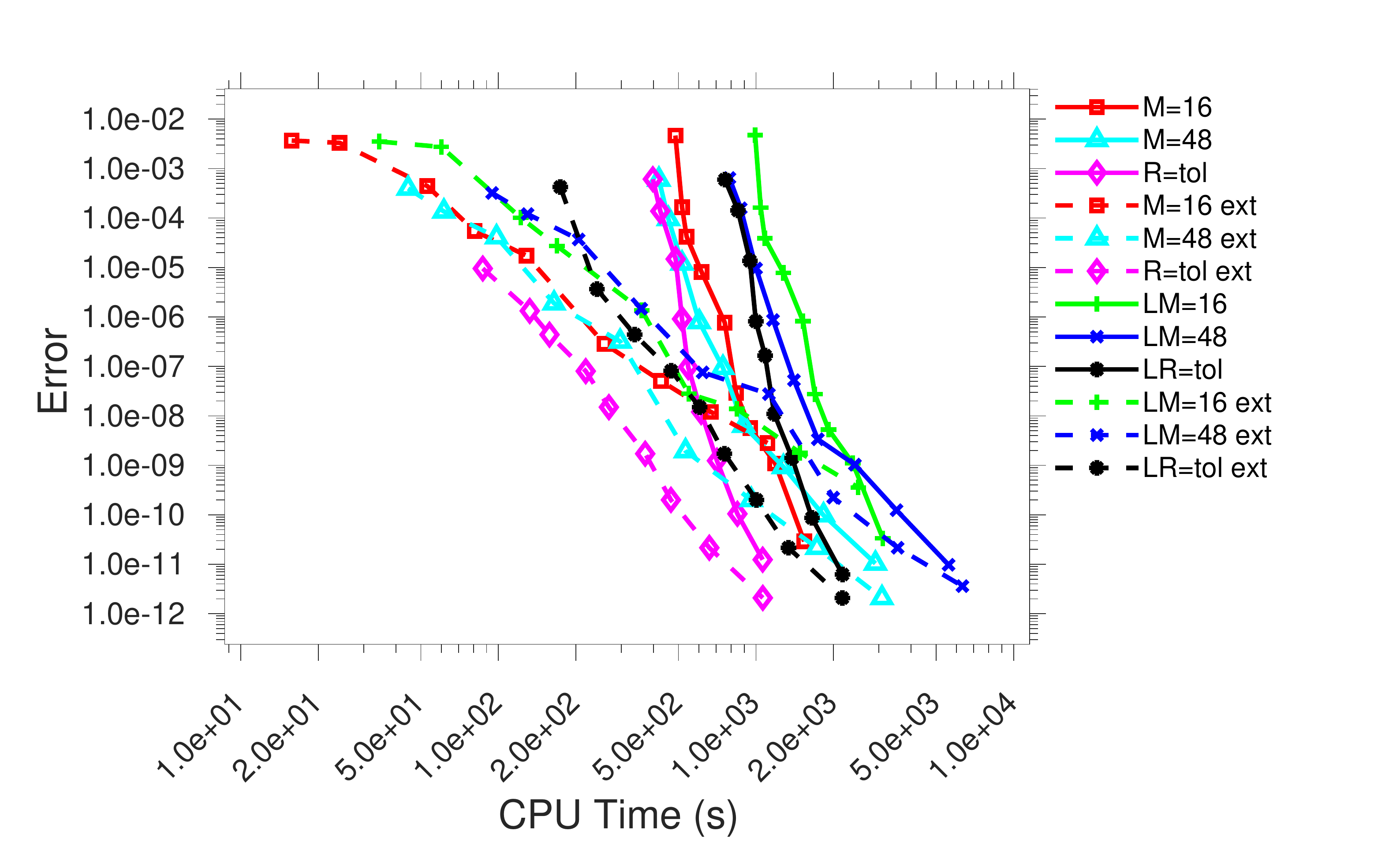}
  \label{fig:qg_ext_fixed_times}
} \\
\subfigure[Timesteps to solution.]{
  \includegraphics[width=4in]{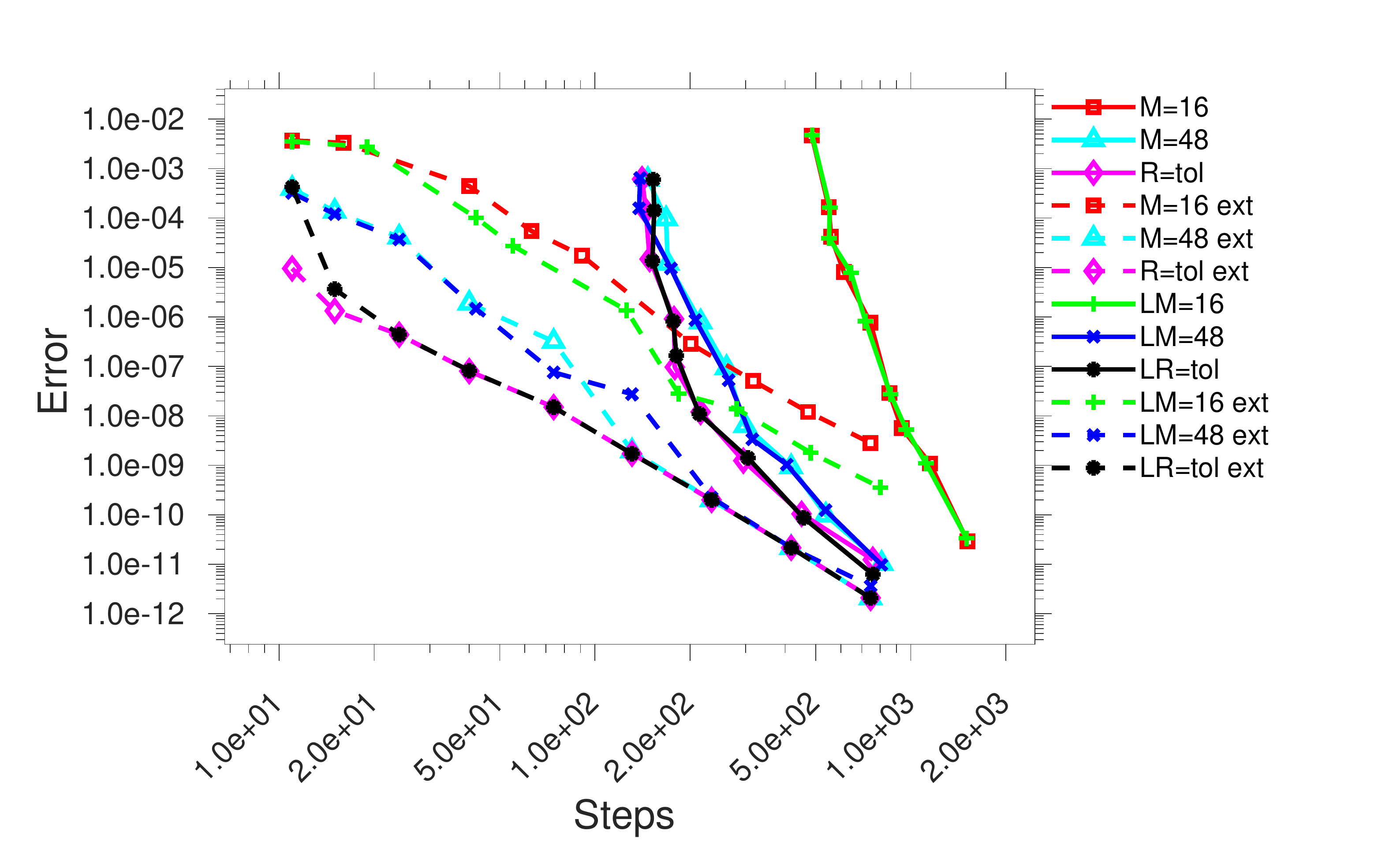}
  \label{fig:qg_ext_fixed_steps}
}
\caption{Extended subspace with fixed basis size results for the quasi-geostrophic problem \eqref{eqn:qgmodel}. Solid lines represent BOROK methods and dashed lines represent the BOROK methods with subspace extension.}
\label{fig:qg_ext_fixed}
\end{figure}

Figure \ref{fig:qg_ext_fixed} compares BOROK against the ROK method with basis extension (as described in \cite{Tranquilli2019}, which does not require solving a least-squares problem). BOROK with basis extension still proves less efficient in the timings, as shown in Figure \ref{fig:qg_ext_fixed_times}. However, the step counts shown in Figure \ref{fig:qg_ext_fixed_steps} demonstrate  that  BOROK and ROK methods have similar stability (and similar stability improvements acquired from basis extension for each). From the timing differences between ROK with and without basis extension and between BOROK with and without basis extension we see very similar timing improvements for each. So, the timing advantage that ROK has for the QG model does not appear to come from the cost of basis extension for BOROK, as the ROK methods see similar performance improvements over BOROK for every basis configuration. Due to the large cost associated with performing Jacobian-vector products and Jacobian-transpose-vector products for this model, the likely cause of this performance discrepancy is the need for an additional Jacobian-transpose-vector product at each iteration of the Lanczos biorthogonalization procedure used in BOROK, which is unnecessary for ROK's Arnoldi iteration. Thus, the doubled number of matrix-vector products dominates the timing results.

\section{Conclusions}
\label{sec:conclusion}

Rosenbrock-Krylov and exponential-Krylov time integration methods are aimed at solving large systems of ordinary differential equations, such as those resulting from the space semi-discretization of partial differential equations in the method of lines framework. These lightly-implicit methods gain efficiency by using of a Krylov subspace approximation to the Jacobian. Since no restarting is allowed, the underlying Arnoldi process becomes computationally expensive when large subspace dimensions are required, e.g., in the solution of stiff systems. 

This work constructs the BOROK family of time integration methods where the embedded Arnoldi iteration is replaced with a biorthogonal Lanczos procedure to obtain Krylov-based Jacobian approximations. This modification allows the BOROK methods to take advantage of the Lanczos procedure's short two-term recurrence, making it possible to use a larger numbers of basis vectors to increase numerical stability of the time integrator, at a moderate computational cost increase. New computational formulas are derived for the Lanczos Jacobian approximation. We revisit the order conditions of Krylov time discretizations to accommodate the new approximate Jacobians. Stability considerations, the link between stability and accuracy, and approaches to adaptively extend the Lanczos bases to further improve stability  are discussed. 

Numerical experiments demonstrate that, while using the same coefficients, BOROK methods retain the same convergence order, accuracy and stability characteristics as the base ROK methods. For nonstiff problems where only very small basis sizes are needed for stability, BOROK methods demonstrate no meaningful performance difference from the base ROK methods. For stiff problems that require large subspaces to approximate the Jacobians, the reduced cost of computing additional basis vectors results in a significant performance advantage for BOROK methods. For problems where  Jacobian-vector products are very expensive (such as those with nearly dense Jacobians) our approach of extending the Krylov space with each method's stage right-hand side is beneficial. However, for our test problem, the performance increase was insufficient to overcome the cost of doubling the number of matrix-vector products as compared to ROK's Arnoldi iteration.

Future work in the class of Krylov time integration methods will address BOROK's requirement for Jacobian-transposed-vector products. A full consideration of the techniques used to produce linear solvers such as BiCGStab, which use of Lanczos biorthogonalization but remove the need for matrix-transpose-vector products at the cost of a small number of extra terms in the recurrence, could bear fruit. Another direction under current pursuit is the extension of Krylov methods to a class of linearly-implicit multistep methods, which does not require the explicit storage of any basis vectors, allowing for the use of nearly any of the wide variety of Krylov-based iterative linear solvers.

\section*{Acknowledgements}

Declaration of Interest: None.

This work has been supported by NSF CCF--1613905, NSF ACI--1709727, AFOSR DDDAS FA9550--17--1--0015, and by the Computational Science Laboratory at Virginia Tech.

This work was performed under the auspices of the U.S. Department of Energy by Lawrence Livermore National Laboratory under Contract DE-AC52-07NA27344.

\section*{References}
%
\bibliographystyle{model1b-num-names}
\bibliography{bibliography,sandu}

\end{document}